\documentclass[12pt, twoside, a4paper]{article} %  OPTIONS:  titlepage

\newcommand{\smallmargin}{0.9in}
\newcommand{\currentmargin}{\smallmargin}

\usepackage[utf8]{inputenc}
\usepackage[T1]{fontenc}
\usepackage{geometry}
\usepackage{lmodern}
\usepackage{amsmath, amssymb, amsthm}
\usepackage{mathtools}
\usepackage{microtype}
\usepackage{xcolor}
\usepackage{hyperref}
\usepackage{titlesec}
\usepackage{enumitem, bbm, lipsum}
\usepackage{fancyhdr}

\mathtoolsset{showonlyrefs} 
\numberwithin{equation}{section}
\allowdisplaybreaks

\setlist[itemize]{itemsep=3pt, topsep=0pt, parsep=0pt}
\setlist[enumerate]{itemsep=3pt, topsep=0pt, parsep=0pt}

\hypersetup{
	colorlinks=true,
	allcolors=blue, 
}

\setlist[itemize]{itemsep=5pt, topsep=5pt, parsep=0pt}
\setlist[enumerate]{itemsep=5pt, topsep=5pt, parsep=0pt}

\theoremstyle{plain} %  plain  mystyle_thm
\newtheorem{theorem}{Theorem}[section]
\newtheorem{corollary}{Corollary}[section] 
\newtheorem{lemma}{Lemma}[section] 
\newtheorem{proposition}{Proposition}[section] 
 
\newtheorem*{lemma0}{Lemma}

\theoremstyle{definition} %   definition       mystyle_def
\newtheorem{definition}{Definition}[section] 

\theoremstyle{remark} %     remark	    mystyle_rmk
\newtheorem{remark}{Remark}

\newcommand{\as}[0]{a.s.}

\newcommand{\wt}[0]{\widetilde}
\newcommand{\sqn}[0]{\sqrt n}
\newcommand{\un}[0]{u_n}

\newcommand{\ucp}[0]{\mathrm{ucp}}

\newcommand{\lra}[0]{ \longrightarrow }
\newcommand{\ra}[0]{ \rightarrow }

\newcommand{\convergence}[1]{ \xlongrightarrow[n\ra \infty]{#1} }
\newcommand{\xucp}[1]{ {#1}\text{-}\ucp}

\newcommand{\cC}[2]{C^{(#1)}_{#2}}
\newcommand{\cB}[2]{B^{(#1)}_{#2}}
\newcommand{\cD}[2]{{\bf Z}^{(#1)}_{#2}}

\newcommand{\Cond}[2]{\wt C^{(#1)}_{#2}}
\newcommand{\CondB}[2]{\wt B^{(#1)}_{#2}}
\newcommand{\CondD}[2]{ {\bf \wt Z}^{(#1)}_{#2}}

\newcommand{\W}[1]{W_{#1}}
\newcommand{\Z}[1]{Z_{#1}}
\newcommand{\Wth}[1]{W^{\theta}_{#1}}
\newcommand{\Xp}[1]{Y'_{#1}}

\DeclareMathOperator{\Esp}{E}
\DeclareMathOperator{\Prob}{P}
\DeclareMathOperator{\Qrob}{Q}
\DeclareMathOperator{\IR}{\mathbb{R}}
\DeclareMathOperator{\IN}{\mathbb{N}}
\DeclareMathOperator{\bF}{\mathcal{F}}
\DeclareMathOperator{\bP}{\mathcal{P}}

\DeclareMathOperator{\erfc}{erfc}
\DeclareMathOperator{\dom}{dom}

\DeclareMathOperator{\D}{D}
\DeclareMathOperator{\Lop}{L}
\DeclareMathOperator*{\sgn}{sgn}

\DeclareMathOperator{\Mills}{Mills}
\DeclareMathOperator*{\MC}{MC}
\DeclareMathOperator*{\xPlim}{\Prob-\lim}

\newcommand{\rd}{\mathrm{d}}
\newcommand{\vd}{\,\mathrm{d}}
\newcommand{\process}[1]{(#1)_{t\ge 0}}
\newcommand{\indic}[1]{\mathbbm{1}_{#1}}
\newcommand{\indicb}[1]{\mathbbm{1}_{\{#1\}}}
\newcommand{\loct}[3]{L^{#2}_{#3}(#1)}
\newcommand{\cloct}[4]{{#1}^{#3}_{#4}(#2)}
\newcommand{\hfprocess}[3]{{#1}^{#2}_{\frac{{#3}}{n}}}

\newcommand{\bigbraces}[1]{ \big( #1 \big) } 
\newcommand{\Bigbraces}[1]{ \Big( #1 \Big) } 
\newcommand{\biggbraces}[1]{ \bigg( #1 \bigg) }

\newcommand{\bigsqbraces}[1]{ \big[ #1 \big] } 
 
\newcommand{\biggsqbraces}[1]{ \bigg[ #1 \bigg] }

\newcommand{\biggcubraces}[1]{ \bigg\{ #1 \bigg\}}

\newcommand{\bigabsbraces}[1]{ \big| #1 \big|}

\newcommand{\qv}[1]{ \langle #1 \rangle }

\newcommand{\hemail}[1]{\href{mailto:{#1}}{\textit{#1}}}  %  \texttt  \textit

\newcommand{\bbreak}[0]{\vspace{1.0em} \noindent}

\usepackage{authblk}

\title{ On the Number of Crossings and Bouncings\\ of a Diffusion at a Sticky Threshold}

\author[1,2]{Alexis Anagnostakis \thanks{\hemail{alexis.anagnostakis@univ-lorraine.fr}} }
\author[3]{Sara Mazzonetto \thanks{\hemail{sara.mazzonetto@univ-lorraine.fr}} }
\affil[1]{Universit\'e de Lorraine, CNRS, IECL, F-57000 Metz, France}
\affil[2]{Universit\'e Grenoble--Alpes, CNRS, LJK, Inria, F-38000 Grenoble, France}
\affil[3]{Universit\'e de Lorraine, CNRS, IECL, Inria, F-54000 Nancy, France}

\newcommand{\shorttitle}{Crossings and bouncings at a sticky threshold}
\newcommand{\shortauthors}{A. Anagnostakis and S. Mazzonetto}

\date{ }

\begin{document}
	
	\maketitle
	
	\begin{abstract}
In this paper, we study the asymptotic behavior of the number of crossings by a one-dimensional diffusion of a threshold where the process exhibits stickiness. 
We distinguish three types of crossings and show that to each type corresponds a distinct asymptotic regime for the respective number of crossings statistic. 
We introduce notions of bouncing as the symmetric counterparts to crossings and show that the corresponding number of bouncings statistics share the same asymptotic properties as their crossings counterparts. We first prove the results for sticky Brownian motion, then extend them to sticky-reflected Brownian motion (where only bouncing is possible) and to sticky diffusions. As an application, we propose consistent estimators for the stickiness parameter of sticky diffusions and sticky-reflected Brownian motion.
	\end{abstract}

\begin{small}
	\emph{Date (of this version):} \quad \today \\
	\emph{2020 Mathematics Subject Classification.}\quad 60J55, 60F05, 60G17, 60J60. \\
	\emph{Key words and phrases.}\quad Sticky Brownian motion, 
	local time approximation,
	occupation time approximation,
	crossing statistics,
	bouncing statistics,
	slow reflection.
\end{small}
	
	\setlength{\abovedisplayskip}{1ex plus 0.5ex minus 0.5ex}
	\setlength{\belowdisplayskip}{1ex plus 0.5ex minus 0.5ex}
	\setlength{\abovedisplayshortskip}{0ex plus 0.5ex} % Can stretch if needed
	\setlength{\belowdisplayshortskip}{1ex plus 0.5ex minus 0.5ex}
	
	\renewcommand{\thefootnote}{}
	\newcommand{\footsection}[1]{\footnotetext{ #1 } }
	
	\footsection{ The authors were partially supported by the Programme Exploratoire Pluridisciplinaire (PEPS) of CNRS Mathématiques. }

	%%%%%%%%%% Section 1
\section{Introduction} \label{sec_intro}
%%%%%%%%%%%%%%%%%%%%

The topic of threshold crossing by a process has been extensively studied in probabilistic literature. 
When the process is a one--dimensional diffusion that solves a classical stochastic differential equation (SDE) with sufficiently regular coefficients, the number of crossings of a threshold $\zeta \in \mathbb{R}$ provides consistent estimations of the local time at $\zeta $ (e.g.~\cite{Aza89,Jac98,Jac2017}; for a definition of local time see Section~\ref{ssec_local_time}).
Similar results have been developed for diffusions with a skew threshold or with discontinuous coefficients, allowing for estimation of localized features of a diffusion, such as the localized diffusion coefficient~(e.g.,~\cite{Florens1988}),
the skew parameter~\cite{Lejay2019,Maz19} and the oscillation jump of the diffusion coefficient~\cite{Maz19}.
Despite these advances, the limit behavior at a sticky threshold remains not well  understood. 

In this paper, we study the limit behavior of the number of the threshold $0$ by the time-discretization of the sticky Brownian motion, as the discretization step vanishes.
The sticky Brownian motion is the diffusion process that behaves like the standard Brownian motion away from $0$ and spends a positive amount of time at $0$ upon contact (see a definition in Section~\ref{ssec_stickyBM}). The set of times at which the process is at $0$ forms a  totally disconnected random set of positive Lebesgue measure. This contrasts with the standard Brownian motion, whose zeroes form an uncountable, totally disconnected set, but of Lebesgue measure $0$ (see, e.g.,~\cite[Proposition III.3.12]{RevYor}).  

To the best of our knowledge, the only established result about number of crossings at the sticky threshold 0 by a sticky Brownian motion $\process{X_t}$ is due to Gikhman.  
The result states that for every $t\ge 0$ as $n$ goes to infinity, the statistic
\begin{equation}
	\sum_{i=1}^{[nt]} \indicb{\hfprocess{X}{}{i-1}\hfprocess{X}{}{i}<0}
\end{equation}
converges in law to some non-trivial discrete random variable (see \cite[\S 8]{Port94}, stated in the next Theorem~\ref{thm_Portenko}). This contrasts with the standard Brownian motion, where this statistic, renormalized by $\sqrt{\pi/2n}$, converges in probability to its local time at $0$.

In the case of sticky Brownian motion, we distinguish between three types of crossings and establish the limit behavior of each respective re-normalized number of crossings statistic. 
Unlike the Brownian motion, we show that the normalizing sequence depends on the type of crossing considered. 
One type of crossings provides a statistic whose renormalized limit matches the one observed in the Brownian case. 
Understanding the limit behavior of all these statistics is a crucial step before tackling the open question of the convergence rates to the local time at a sticky threshold. 
For non-sticky processes, these have been obtained in~\cite[Theorem~1.2]{Jac98} for classical diffusions and to~\cite{Maz19} for skew threshold
diffusions.

Ultimately, establishing such rates would enable us to obtain similar rates for the stickiness parameter estimators established here and in   \cite{Anagnostakis2022,AnagnostakisM2023}, which are built on such approximations. 
To highlight the importance of the estimation problem, let us mention a few applications of sticky diffusions. 
These processes are used in finance to model price dynamics on asset subject to takeover offers (see~\cite{criens2022separating}). In quantum mechanics, to model motions of particles near a source of emission (see \cite{DavTru}). In classical mechanics, to model motions of coarse particles in colloids (see \cite{BouRabee2020} and references therein). And in epidemics, to model concentrations of pathogens in an organism (see~\cite{calsina2012steady}).
During an infection, these concentrations are random and significant, whereas off-infection, they keep close to zero. 

Another contribution of this paper is the introduction of the symmetric counterpart of crossings, called bouncings, defined here as rebounds of the time-discretization of the process at some threshold. 
As with crossings, we distinguish three kinds of bouncing and prove that the respective number of bouncings has the same asymptotic properties as the number of crossings of the same kind. We combine results on bouncings and crossings to establish results on the sticky-reflected Brownian motion (see Section~\ref{ssec_stickyrefl} for a definition), where only bouncing behavior is possible, and also on some sticky It\^o diffusions. Our findings allow for the construction of a consistent stickiness parameter estimator, which we numerically show to converge at a comparable or faster rate than existing estimators in~\cite{Anagnostakis2022, AnagnostakisM2023}. 

\bbreak
\emph{Outline.}\quad The paper is organized as follows.  
In Section~\ref{sec_preliminaries}, we introduce the sticky Brownian motion along with several notions and results, which are useful for this paper.
In Section~\ref{sec_main_result}, we state the main results of this paper, namely, the limit behavior of the considered
number of crossings and bouncings statistics (Theorems~\ref{thm_main}, \ref{thm_main_R0}) and the consistency of a stickiness parameter estimator based on these (Proposition~\ref{prop_estimation}). 
In Section~\ref{sec_proof_B}, we prove these results.
In Section~\ref{sec_SID}, we extend the results to the sticky-reflected Brownian motion and to smooth Itô diffusions with a sticky threshold. 
Section~\ref{sec_numexp} is dedicated to numerical experiments on the stickiness parameter estimator devised in this paper.

In the Appendix we prove several useful results for our paper.   
More precisely, in Appendix~\ref{app_asymptotics}, we prove asymptotic results on the sticky Brownian motion transition kernel.
In Appendix~\ref{app_reflection_sticky}, we prove a reflection principle at 0 for the sticky Brownian motion.

%%%%%%%%%% Section 2
\section{Preliminary notions and results}
\label{sec_preliminaries}
%%%%%%%%%%%%%%%%%%%%

In this section we state preliminary notions and results useful for the rest of the paper.
We begin by defining the sticky Brownian motion, recall its probability transition kernel and its space-time scaling property.
Then, we recall two notions of convergence. The first is the convergence uniform in time, in probability, in which the
local time approximation results are expressed and which retains the functional character of the approximation.
The second is the conditional convergence in probability, in which the estimation results are expressed.
The estimators are indeed consistent only on the event that the threshold of interest ($0$ in our case) is hit by the process.
Last, we state convergence results, in particular a local time approximation result previously established in~\cite{AnagnostakisM2023}, to which 
the proofs of our main results are reduced to.

\subsection{Local time}
\label{ssec_local_time}

We start by recalling a definition of the local time, which 
plays a fundamental role in our analysis.
Indeed, in many cases, it 
will be the limiting process of our statistics of interest.

The local time of a semi-martingale, $X$,  
is the random field $(\loct{X}{a}{t};\; a\in \IR,\,t\ge 0) $, defined for all 
$t\ge 0 $, $a\in \IR $ as the term $\loct{X}{a}{t} $
so that the following equation holds (see~\cite[Theorem~1.2]{RevYor})
\begin{equation}
	|X_t - a| = |X_0 - a| + \int_{0}^{t} \sgn(X_s - a) \vd X_s + \loct{X}{a}{t}.
\end{equation}
In the above equation we used the convention $\sgn(x)=\indicb{x>0}-\indicb{x\le 0} $. 

For all $a\in \IR $, $t \mapsto \loct{X}{a}{t}$ is increasing and is fully supported on $\{t\ge 0:\; X_t = a\} $. 
This means that 
\begin{equation}
	\label{eq_txt_support_loctime}
	\begin{aligned}
		\forall t &\ge 0:
		& \int_{0}^{t}\indicb{|X_s-a| \not = 0} \vd \loct{X}{a}{s}&= 0.
	\end{aligned}
\end{equation}

The local time field can alternatively be defined as the almost sure limit (see~\cite[Corollary~VI.1.9]{RevYor}): 
\begin{equation}
	\label{eq_txt_loctime_char_2}
	\begin{aligned}
		\forall &t\ge 0,\, a\in \IR,
		& \text{almost surely}:&
		&\loct{X}{a}{t}
		&= \lim_{\epsilon \ra 0} \frac{1}{\epsilon}\int_{0}^{t} \indicb{X_s \in [a,a + \epsilon)} \vd \qv{X}_s.
	\end{aligned}
\end{equation}

\subsection{The sticky Brownian motion}
\label{ssec_stickyBM}

We now provide a definition of sticky Brownian motion as a general linear diffusion on the real line described by scale function and speed measure. We explain what is the role of the stickiness parameter without discussing further what are scale function and speed measure. Instead, we provide other equivalent characterizations of the process which are useful in the proofs of our results.

The sticky Brownian motion of stickiness parameter $\rho>0$ is the diffusion process on $\IR $, on natural scale, with speed measure 
$m(\rd x) = \vd x + \rho \delta_{0}(\rd x) $ (see, e.g., \cite[p.123-124]{BorSal}). 
The parameter $\rho>0 $, called stickiness parameter, expresses the propensity of the process to stick at $0$. The higher it is, the more time the process spends on average at $0$.
The asymptotic case $\rho = 0 $ corresponds to the standard Brownian motion and
$\rho=\infty $ to the Brownian motion with an absorbing boundary at $0$.
In this paper we deal only with the case $\rho \in (0,\infty) $.
The %standard 
Brownian motion case is known and the absorbing case is trivial.

\bbreak
We now provide some useful equivalent characterizations of sticky Brownian motion. 
Let $\mathcal P_x = (\Omega, \bF, \process{\bF_t}, \Prob_x)$ be a filtered probability space. 
The subscript $x $ in $\mathcal P_x $ is a notational choice to indicate the starting point of the process. 
The filtrations $\process{\mathcal F_t}$ as well as all filtrations in this paper are assumed to satisfy the usual conditions (right-continuity and completeness).

The following hold (see e.g.~\cite[Sections~5.1, 5.2]{ItoMcKean96} and \cite[Theorem 47.1 and Remark~(ii), p.277]{RogWilV2}):
\begin{enumerate}
	[label={\upshape (t\arabic*)}]  
	\item \label{item_sticky_P1} Let $X$ be a sticky Brownian motion of stickiness parameter $\rho$, defined on $\mathcal P_x$ such that $\Prob_x$-\as, $X_0=x $. 
	(In particular, $X$ is $\process{\bF_t}$-adapted.)
	There exists a Brownian motion $\Z{}$, defined on an extension of $\mathcal P_x $, such that $X=\process{\Z{\gamma(t)}} $, where the time--change $\gamma$ is the right-inverse of 
	\begin{equation} \label{eq:gamma}
		\begin{aligned}
			A(t) 
			%&:= \int_{\mathbb R} \loct{Z}{y}{t} m(\!\vd y)
			&:= t + \rho \loct{\Z{}}{0}{t}, 
			&  t&\ge 0,
		\end{aligned}
	\end{equation} 
	given by $\gamma(t):=\inf\{s>0 \colon A(s)>t\}$ 
	and where $\loct{\Z{}}{0}{t}$ is the right local time at $0$ of the process $\Z{}$.% (see \cite[p.277]{RogWilV2}).
	
	\item \label{item_sticky_P1_rev} Let $Z $ be a standard Brownian motion, defined on the probability space $ \mathcal P_x$, such that $\Prob_x $-\as: $Z_0=x $. Then, the process 
	$X := \process{Z_{\gamma(t)}} $ with $\gamma $ defined in~\ref{item_sticky_P1}, is 
	a sticky Brownian motion of stickiness parameter $\rho $, and, $\Prob_x $-\as: $X_0 = x $.
\end{enumerate}	
Also, the following SDE characterization  holds. 
\begin{enumerate}
	[label={\upshape (p\arabic*)}] 
	\item \label{item_sticky_P2} Let $X$ be a sticky Brownian motion of stickiness parameter $\rho$, defined on the probability space $\mathcal P_x 
	%= (\Omega, \mathcal F, \Prob_x)
	$, such that $\Prob_x$-\as, $X_0=x $.
	There exists a Brownian motion $\W{} $ defined on an extension
	of $\mathcal P_x $ such that $(X,W) $ solves
	\begin{equation}
		\label{eq_txt_sticky_pathwise}
		\begin{dcases}
			\vd X_t = \indicb{X_t \not = 0} \vd \W{t}, \quad X_0=x,\\
			\indicb{X_t = 0} \vd t = \rho \vd \loct{X}{0}{t}.	
		\end{dcases}
	\end{equation}
	
	\item \label{item_sticky_P2_rev} The system \eqref{eq_txt_sticky_pathwise} has a jointly unique weak solution \cite[Theorem 1]{EngPes} and it is a sticky Brownian motion of stickiness parameter $\rho $.
\end{enumerate}
A proof of the last two statements for some sticky diffusions can be found respectively in \cite[Proposition~4.2]{Anagnostakis2022} for \ref{item_sticky_P2} and in \cite[Theorem~4.1]{Anagnostakis2022} for \ref{item_sticky_P2_rev}. 
For sticky Brownian motion these proofs are contained in the proof of~\cite[Theorem~1]{EngPes}.

For an historical overview of the sticky Brownian motion, see~\cite{Pesk2014}. 
For results on this process, see, e.g., \cite{Ami,Anagnostakis2022,AnagnostakisM2023,DavTru,Howitt2007,Sal2017}.
For applications, see~\cite{Hol-Cer2020} and references therein.

\begin{remark}[Occupation time and local time]
	%Note that 
	The second line of the system~\eqref{eq_txt_sticky_pathwise} states that the occupation time of the process at $0$ is proportional to its local time at $0$. This is related to the fact that the process spends a positive amount of time at the threshold $0$. 
\end{remark}

\subsection{Notions of convergence}

The results on local time approximation are formulated in terms of the types of convergence introduced below. For convenience, let 
be $(A_n)_{n\ge 0} $ a sequence of real-valued processes defined on the probability space 
$(\Omega, \mathcal F, \Prob) $.

\begin{definition} \label{def:ucp}
	We say that $(A^{n})_{n\ge1} $
	\emph{converges locally uniformly in time, in probability} 
	%or \ucpB{} 
	to $A^{0}$ if 
	\begin{equation}
		\forall t\ge 0:\;
		\sup_{s\le t} \bigabsbraces{A^{n}_s - A^{0}_s} \convergence{\Prob} 0.
	\end{equation}
	We denote this convergence with 
	\begin{equation}
		A^{n} \convergence{\xucp{\Prob}} A^{0} .
	\end{equation}
\end{definition}

The following result gives a sufficient condition for ucp convergence to occur.

\begin{lemma}[cf.~\cite{JacPro}, \S 2.2.3] 
	\label{lem_ucp_convergence_condition}
	Assume that $A^n$, for all $n\ge 1 $, and $A^{0}$ have increasing paths and $A^{0}$ is continuous. If there exists $D$ dense in $[0,\infty)$ such that
	$
	A^{n}_{t} \overset{\Prob}{\longrightarrow} A^{0}_t,\; \  \forall t \in D, 
	$
	then 
	%	\; \Longrightarrow \;
	$
	A^{n} \convergence{\xucp{\Prob}} A^{0}.
	$
\end{lemma}

For the estimation result we adopt
the notion of \emph{conditional convergence in probability}. 
The reason is that the estimators we consider are consistent conditionally on 
the event that the threshold of interest (assumed located at $0$) is reached. 
For an event $\mathcal H $ with $\Prob(\mathcal H)\neq 0$, let $\Prob^{\mathcal H}(\cdot) := \Prob(\cdot| \mathcal H) $ be the conditional probability on $\mathcal H $.

\begin{definition}
	\label{def_conditional_convergence_in_probability}
	We say that the sequence of random variables $(X_{n})_{n\ge 1} $ converges to the random variable $X_{0}$ in probability, conditionally on the event $\mathcal H $ with positive probability, 
	if $X_{n}\lra X_{0} $ in $\Prob^{\mathcal H}$-probability.
	We denote this convergence with
	\begin{equation}
		X_{n} \convergence{\Prob^{\mathcal H}} X_{0}.
	\end{equation}
\end{definition}

\subsection{Existing useful asymptotic results}

For establishing the limit behavior of the number of crossings statistics we use the following existing results 
on the sticky Brownian motion: an approximation of the local time, an approximation of the occupation time at $0$, and a result on discrete martingales. 
We conclude the section by recalling a previously established limit in law of the number of crossings (of a precise type, type $0$, described in Section~\ref{sec_main_result}), that we do not use in the paper because we do not find the proof.

\begin{proposition}[\cite{AnagnostakisM2023}~Proposition~6.3]
	\label{prop_SMA_1}
	Let $X$ be the sticky Brownian motion of stickiness parameter $\rho>0 $, defined on the filtered probability space
	$(\Omega, \bF, \process{\bF_t},\Prob_x) $ such that $\Prob_x $-\as, $X_0=x $.
	Let $m_{\sqn \rho} $ be the measure defined for all $x\in \IR $ by
	\begin{equation} \label{def_mn}
		m_{\sqn \rho}(\rd x) = \vd x + \sqn \rho \, \delta_{0}(\rd x)
	\end{equation}
	and $(g_{n})_{n}$ a sequence of measurable real functions such that for all $x\in \IR$ 
	\begin{multline}
		\label{eq_connd_gn_aggreg}
		\lim_{n\lra \infty} \Bigbraces{ \frac{g_{n}^2(\sqn x)}{n} + 
			\frac{m_{\sqn \rho}(g^{2}_n)}{\sqn}  \\
			+  \frac{ \bigbraces{\int_{-\infty}^{+\infty} |x| g_n(x) \vd x} \bigbraces{1 + \log(n)} g_n(\sqn x)}{n}
			+   \frac{{  \bigbraces{1 + \log(n)} m_{\sqn \rho}(|g_n|)}}{\sqn}}
		= 0
	\end{multline}
	and $\lim_{n\to \infty}m_{\sqn \rho}(g_n) = M$.
	Then, for all $t \geq 0$,
	\begin{equation}\label{eq_the_case_a_05_plain}
		\frac{1}{\sqn} \sum_{i=1}^{[nt]}g_n(\sqn \hfprocess{X}{}{i-1}) \convergence{\Prob_x} M \loct{X}{0}{t}.
	\end{equation}
	Also, if~\eqref{eq_connd_gn_aggreg} holds and $\sup_n \bigbraces{m_{\sqn \rho}(|g_n|)} <\infty$ then the above convergence is localy uniform in time, in $\Prob_x$-probability.
\end{proposition}
In the above proposition, as well as in the entire paper, we use the notation 
\begin{equation} \label{notaz_m}
	m_{\sqn \rho}(g) =  \int_{\mathbb{R}} g(x)  m_{\sqn \rho} (\rd x) = \int_{-\infty}^{+\infty} g(x) \vd x + \sqrt{n}\rho \, g(0)
\end{equation}
for any integrable function $g$.

\begin{lemma}[\cite{AnagnostakisM2023}, Lemma~3.7]
	\label{thm_Altmeyer}
	Let $X$ be a sticky Brownian motion of stickiness parameter $\rho>0$, defined on the filtered probability space
	$(\Omega,\bF,\process{\bF_t},\Prob_x) $ such that $\Prob_x $-\as, $X_0=x $.
	Then, it holds that, for all intervals $U$ (including singletons),
	\begin{equation}
		\label{eq_lem_SOSBM_occtime_approximation}
		\frac{1}{n} \sum_{i=1}^{[n\cdot]} \indicb{\hfprocess{X}{}{i-1} \in U}
		\xrightarrow[n\lra \infty]{\Prob_x\text{-}\ucp} \int_{0}^{\cdot} \indicb{X_s \in U} \vd s. 
	\end{equation}
	In particular, taking $U=\{0\} $, it yields that
	\begin{equation}\label{eq_proof_occtime_statistic_convergence}
		\frac{1}{n} \sum_{i=1}^{[n\cdot]} \indicb{\hfprocess{X}{}{i-1} = 0} \xrightarrow[n\lra \infty]{\Prob_x\text{-}\ucp}
		\int_{0}^{\cdot} \indicb{X_s = 0} \vd s =  \rho \loct{X}{0}{}.
	\end{equation}
\end{lemma}

See also \cite{Altmeyer2023}, for the approximation of smooth
occupation time functionals.

\begin{lemma}[\cite{Jacod1993}, Lemma~9]
	\label{lem_Genon_Catalot}
	Let $\mathcal P = (\Omega, \bF, \Prob) $ be a probability space, $((\bF^{n}_{i})_{i})_n $ a family of discrete filtrations on $\mathcal P $, $(\chi^{n}_i)_{i,n}$ a family of random variables
	on $\mathcal P $ such that for all $(i,n) $, $\chi^{n}_i$ being $\bF^{n}_{i}$-measurable and $U$ a random variables on $\mathcal P $. The following two conditions imply $\sum_{i=1}^{n} \chi^{n}_{i}\convergence{\Prob} U$:
	\begin{enumerate}
		[label={ (\roman*)}]
		\item $\sum_{i=1}^{n} \Esp \left( \chi^{n}_{i} \big| \bF^{n}_{i-1} \right) \convergence{\Prob} U$,
		\item $\sum_{i=1}^{n} \Esp \left( \left(\chi^{n}_{i}\right)^{2} \big| \bF^{n}_{i-1} \right) 
		\convergence{\Prob} 0$.
	\end{enumerate}
\end{lemma}

\begin{theorem}
	[Efimenko-Portenko (1989), see~\cite{Port94}]
	\label{thm_Portenko}
	We consider the setting of Lemma~\ref{thm_Altmeyer}. 
	For all $t>0 $, 
	%$\cC{0}{n,t}(X) $ 
	$\sum_{i=1}^{[nt]} \indicb{\hfprocess{X}{}{i-1}\hfprocess{X}{}{i}<0}$
	converges in law to some discrete random variable $Z_t $
	with values in $\IN_0 $, i.e., there exists $(b_k(t);\; k\in \IN_0,\, t\ge 0) $ such that $\Prob_0 \left( Z_t = k \right) = b_k(t) $, for all $k\in \IN_0 $.
	The probabilities $(b_k(t);\; k\in \IN_0,\, t\ge 0) $ are defined via their Laplace transforms 
	\begin{equation}
		\begin{aligned}
			\int_{0}^{\infty} \lambda e^{-\lambda t} b_k(t) \vd t &= 
			\frac{A_{\lambda}}{1/\rho^{2} + A_{\lambda}}
			\left( \frac{1/\rho^{2}}{1/\rho^{2}+A_{\lambda}} \right)^{k},
			& \forall &k\in \IN_0,
			& \forall & t\ge 0,  
		\end{aligned}
	\end{equation}
	with $A_{\lambda}= \sqrt{\lambda} \left( \sqrt \lambda + \sqrt{2}/\rho \right) $. 
\end{theorem}

%%%%%%%%%% Section 3
\section{Main results}
\label{sec_main_result}
%%%%%%%%%%%%%%%%%%%%

We consider three types of behaviors on the time interval $[t,t+s]$ of the process $X$ at the sticky threshold. For simplicity, we assume the threshold to be located at $0$.

The types of crossings are:
\begin{enumerate}
	\item The \textit{crossing of type 0}: the process
	crosses from one open half-plane to the other. For example if $X_t <0 $ and
	$X_{t+s}>0 $. %, with $t\geq 0, s>0$.
	\item The \textit{crossing of type 1}: the process
	crosses from one open or closed half-plane to its complementary. For example if $X_t \le0 $ and
	$X_{t+s}>0 $ or if $X_t <0 $ and $X_{t+s} \ge 0 $.
	\item The \textit{crossing of type 2}: the process crosses from
	one closed half-plane to the other. For example $X_t \le 0 $ and
	$X_{t+s}\ge 0 $.
\end{enumerate}

The types of bouncings are:
\begin{enumerate}
	\item The \textit{bouncing of type 0}: the process
	bounces from $(0,\infty) $ back into $(0,\infty) $ after hitting $0$. For example if $X_t >0 $, $X_{t+s}>0 $,
	and for some $h\in[0,s] $, $X_{t+h}=0 $.
	\item The \textit{bouncing of type 1}: the process
	bounces from the closed half-plane to the open one or vice-versa. For example if $X_t > 0$ and
	$X_{t+s}\ge 0 $ or if $X_t >0 $ and $X_{t+s} \ge 0 $, and for some $h\in[0,s] $, $X_{t+h}=0 $.
	\item The \textit{bouncing of type 2}: the process bounces from $[0,\infty) $ to $[0,\infty) $. For example $X_t \ge 0 $, $X_{t+s}\ge 0 $, and for some $u\in[t,t+s] $, $X_{u}=0 $.
\end{enumerate}

We now define the crossing statistics based on discrete observations of the process. Assume we are given uniformly spaced in time observations of a process $X$ over the interval $[0,t]$, taken at times ${i/n : i = 0, 1, \dots, [nt]}$.
That is, we are given $(X_{i/n})_{i\le [nt]}$.

For $j\in \{0,1,2\}$, the \textit{number of crossings statistic of type $j$} 
is the number of times the time-discretization $(X_{i/n})_{i\le [nt]} $ of the process executes a crossing of type $j$ at $0$.
In particular, for all $t \in (0,\infty),n\in \mathbb{N} $,
\begin{itemize}
	\item[] $\cC{0}{n,t}(X) := 
	\sum_{i=1}^{[nt]} \indicb{\hfprocess{X}{}{i-1}\hfprocess{X}{}{i}<0} $,
	\item[] $\cC{1}{n,t}(X) := 
	\sum_{i=1}^{[nt]} \bigbraces{ \indicb{\hfprocess{X}{}{i-1}\hfprocess{X}{}{i}\le 0}
		- \indicb{\hfprocess{X}{}{i-1}=\hfprocess{X}{}{i} = 0} } $,
	\item[] $ \cC{2}{n,t}(X) := 
	\sum_{i=1}^{[nt]} \indicb{\hfprocess{X}{}{i-1}\hfprocess{X}{}{i}\le 0}$.
\end{itemize}
Similarly, for all $t \in (0,\infty),n\in \mathbb{N} $, we define the number of bouncings statistics as
\begin{itemize}
	\item[] $\cB{0}{n,t}(X) := 
	\sum_{i=1}^{[nt]} \indic{U^{n}_{i}(X)} \indicb{\hfprocess{X}{}{i-1}\hfprocess{X}{}{i}>0} $,
	\item[] $\cB{1}{n,t}(X) := 
	\sum_{i=1}^{[nt]} \indic{U^{n}_{i}(X)} \bigbraces{ \indicb{\hfprocess{X}{}{i-1}\hfprocess{X}{}{i}\ge 0}
		- \indicb{\hfprocess{X}{}{i-1}=\hfprocess{X}{}{i} = 0} } $,
	\item[] $ \cB{2}{n,t}(X) := 
	\sum_{i=1}^{[nt]} \indic{U^{n}_{i}(X)} \indicb{\hfprocess{X}{}{i-1}\hfprocess{X}{}{i}\ge 0}$.
\end{itemize}
where, for all $n,i $, $U^{n}_{i}(X) $ 
is the event defined as
\begin{equation}
	U^{n}_{i}(X) = \biggcubraces{
		\exists s \in \biggsqbraces{\frac{i-1}{n},\frac{i}{n}}: \; X_s = 0}.
\end{equation}

We note that these statistics are different when the process spends a positive amount of time at $0$ with positive probability. 
Indeed, if we define the \emph{difference statistics} as
\begin{align}
		\cD{1}{n,t}(X) 
		& := \cC{1}{n,t}(X) - \cC{0}{n,t}(X)
		=  \cB{1}{n,t}(X) - \cB{0}{n,t}(X)
		\\
		&= \sum_{i=1}^{[nt]} \indicb{\hfprocess{X}{}{i-1}=0;\;\hfprocess{X}{}{i}\not= 0} + \sum_{i=1}^{[nt]} \indicb{\hfprocess{X}{}{i-1}\not = 0;\;\hfprocess{X}{}{i}= 0},\\
		\cD{2}{n,t}(X) 
		&:= \cC{2}{n,t}(X) - \cC{1}{n,t}(X)
		=  \cB{2}{n,t}(X) - \cB{1}{n,t}(X) =  \sum_{i=1}^{[nt]} \indicb{\hfprocess{X}{}{i-1}=\hfprocess{X}{}{i}= 0},
\end{align}
then, for all $t,n $ and $i=1,2 $, for sticky Brownian motion $X $ it holds that
$\cD{i}{n,t}(X)>0 $ with positive probability. If $X$ is a Brownian motion, then $\cD{i}{n,t}(X)=0$ almost surely.

One may say that $\cB{j}{}$ are not statistics, since they depend on the events $(U_{n,i})_{n,i} $ which are not in the $\sigma $-algebra generated by high-frequency samples of the process. However, they are related to the \emph{difference statistics} $\cD{1}{},\cD{2}{} $.

\subsection{Limit behavior of the statistics} % $\cC{j}{}{},\cB{j}{}{}$}

The main result of this paper is the following. 

\begin{theorem}
\label{thm_main}
Let $X$ be the sticky Brownian motion of stickiness parameter $\rho>0$ defined on the filtered probability space $(\Omega, \bF, \process{\bF_t},\Prob_x) $ such that $\Prob_x$-\as, $X_0 = x$ (in particular, $X$ is $\process{\bF_t}$-adapted),  
and let $\loct{X}{0}{} $ be the right local time at $0$ of $X$.
Then, the following convergences holds, 
\begin{enumerate}
	[label={ (\roman*)}]
	\item \label{item_strict} for every diverging sequence $\un$, it holds that:  
	$ \cC{0}{n,\cdot}(X)/\un  \convergence{\Prob_x\text{-}\ucp} 0 $, 
	\item 
	$ \cD{1}{n,\cdot}(X) /\sqn \convergence{\Prob_x\text{-}\ucp}
	4 \sqrt{\frac{2}{\pi}} \loct{X}{0}{} $ 
	and $ \cC{1}{n,\cdot}(X) /\sqn \convergence{\Prob_x\text{-}\ucp}
	4 \sqrt{\frac{2}{\pi}} \loct{X}{0}{} $, 
	\label{item_large} 
	\item 
	$ \cD{2}{n,\cdot}(X) /n \convergence{\Prob_x\text{-}\ucp} \rho \loct{X}{0}{} $
	and 
	$ \cC{2}{n,\cdot}(X) /n \convergence{\Prob_x\text{-}\ucp} \rho \loct{X}{0}{}$, \label{item_stays}
\end{enumerate}
where as per~\eqref{eq_txt_sticky_pathwise}, $\rho \loct{X}{0}{t} = \int_{0}^{t} \indicb{X_s = 0} \vd s  $, for all $t\ge 0 $.
\end{theorem}
We note that in the case of the standard Brownian motion $W$, all three statistics have
the same non-trivial limit for the same normalizing sequence $\sqrt{n} $.
Indeed, from, e.g.,~\cite[Theorem~1.1]{Jac98}, 
\begin{equation}
\label{rel_standard_Brownian_motion}
\xPlim_{n\lra \infty} \left( \frac{\cC{0}{n,t}(W)}{\sqn} \right)
= 
\xPlim_{n\lra \infty} \left( \frac{\cC{1}{n,t}(W)}{\sqn} \right)
=
\xPlim_{n\lra \infty} \left( \frac{\cC{2}{n,t}(W)}{\sqn} \right)
= \sqrt{\frac{2}{\pi}} \loct{W}{0}{t}.
\end{equation}
Also, the rates of convergence are known. 
In~\cite[Theorem~1.2]{Jac98}, a central limit theorem is proven  
for this convergence, with a rate of $n^{1/4} $.

Unlike the standard Brownian motion, we observe that there is a factor $4$ in the limits of $\cC{1}{n\cdot}/\sqn$ and $\cD{1}{n,\cdot}/\sqn $ for the sticky Brownian motion (see~Theorem~\ref{thm_main}\ref{item_large}).
The factor $4$ comes from the fact that $\cD{1}{n,t}/\sqn$ rewrites as sum of $4$ different terms, all converging to $\sqrt{2/\pi} \loct{X}{0}{t}$. 
Indeed, 
\begin{equation}
\begin{aligned}
	\frac1\sqn \cD{1}{n,t}(X) =& \frac1\sqn\sum_{i=1}^{[nt]}  \indicb{\hfprocess{X}{}{i-1}=0 ;\;\hfprocess{X}{}{i}> 0} 
	+
	\frac1\sqn\sum_{i=1}^{[nt]} \indicb{\hfprocess{X}{}{i-1}=0 ;\;\hfprocess{X}{}{i} < 0}
	\\&+ 
	\frac1\sqn\sum_{i=1}^{[nt]} \indicb{\hfprocess{X}{}{i-1} > 0;\;\hfprocess{X}{}{i}= 0}
	+
	\frac1\sqn\sum_{i=1}^{[nt]} \indicb{\hfprocess{X}{}{i-1} < 0;\;\hfprocess{X}{}{i}= 0}
	.
\end{aligned}
\end{equation}
Each term converges to $0$ if the process is a standard Brownian motion $W$.
Instead, for the sticky Brownian motion, the limit is non trivial and it corresponds to the case of the renormalized number of crossings statistic for Brownian motion: $\cC{0}{n,\cdot}(W)/\sqn$.

Regarding the bouncing statistics, we prove that: 

\begin{theorem}
\label{thm_main_R0}
Consider the setting of Theorem~\ref{thm_main}. Then
\begin{enumerate}
	[label={ (\roman*)}]
	\item \label{item_strict_R0} for every diverging sequence $\un$, 
	$\cB{0}{n,\cdot}/\un$ has the same limit as $\cC{0}{n,\cdot}(X) /\un$,
	\item 
	$\cB{j}{n,\cdot}/n^{j/2}$ has the same limit as $\cC{j}{n,\cdot}(X) /n^{j/2}$, for $j=1,2$. 
	\label{item_large_R0}
\end{enumerate}
\end{theorem}

The difference statistics $\cD{\cdot}{}$ are the leading parts of number of crossings and bouncings statistics of type $1$ and $2$.

The proof of the above results relies, by virtue of Lemma~\ref{lem_Genon_Catalot}, on the following limit behavior of the conditional version of the statistics $\cC{0}{}$, $\cC{1}{}$, and $\cC{2}{}$. 
These are the statistics $\Cond{0}{},\Cond{1}{},\Cond{2}{} $ defined for all $n,t $ as 
\begin{itemize}
\item[] $\Cond{0}{n,t}(X) := \sum_{i=1}^{[nt]} \Esp_x \bigbraces{ \indicb{\hfprocess{X}{}{i-1}\hfprocess{X}{}{i}<0} \big| \bF_{\frac{i-1}{n}} } $,
\item[] $\Cond{1}{n,t}(X) := \sum_{i=1}^{[nt]} \Esp_x \bigbraces{ \indicb{\hfprocess{X}{}{i-1}\hfprocess{X}{}{i}\le 0}
	- \indicb{\hfprocess{X}{}{i-1}=\hfprocess{X}{}{i} = 0} \big| \bF_{\frac{i-1}{n}} } $,
\item[] $ \Cond{2}{n,t}(X) := \sum_{i=1}^{[nt]} \Esp_x \bigbraces{ \indicb{\hfprocess{X}{}{i-1}\hfprocess{X}{}{i}\le 0} \big| \bF_{\frac{i-1}{n}} }$.
\end{itemize}
For these, the following result holds.

\begin{proposition}
\label{thm_main_B}
Consider the setting of Theorem~\ref{thm_main}. 
Then %following convergences holds:
\begin{enumerate}
	[label={ (\roman*)}]
	\item \label{item_strict_B} $\Cond{0}{n,\cdot}(X) \convergence{\Prob_x\text{-}\ucp}
	({1}/{\rho}) \, \loct{X}{0}{} $, 
	\item \label{item_large_B} $\Cond{j}{n,\cdot}(X)/n^{j/2}$ has the same limit as $\cC{j}{n,\cdot}(X)/n^{j/2}$ for $j=1,2$.
\end{enumerate}
\end{proposition}

The proof of Theorem~\ref{thm_main}\ref{item_strict} is done by bounding the distance between $ \Cond{0}{}$ and $\cC{0}{} $ and showing that
both the distance and the limit of $\Cond{0}{}(X)$, if renormalized by $1/\un $, vanish in probability, as $n\lra \infty $.

Let us note that the only case, for which the limits of the expected and non-expected versions of
the number of crossings statistics are not necessarily the same, is for the Type-$0$ crossings (strict crossings).
In fact, they cannot be the same since $\cC{0}{n,t}(X) $ and its limit take values in $\IN \cup (+\infty)$, while the law of $(1/\rho) \loct{X}{0}{t} $, the limit of $\Cond{0}{n,t}(X) $, is continuous. 
The discrete limit law of $\cC{0}{n,t}(X)  $ is described in~\cite[\S 8]{Port94}, which is restated here in Theorem~\ref{thm_Portenko}.
Let us also mention that Theorem~\ref{thm_main}\ref{item_strict} can also be inferred directly from Theorem~\ref{thm_Portenko}.

\subsection{Estimation}

In~\cite{Anagnostakis2022,AnagnostakisM2023} an estimator was proposed based on a test function $g$ and
a normalizing sequence $\un $.
More precisely, let $g $ be a bounded integrable function that satisfies $g(0)=0 $ and $(\un)_n $ a sequence that
as $n\lra\infty $: $\un \lra \infty $, $\un/n \lra 0 $, $\mathcal H_t $ be the event $\{\tau^{X}_0<t\} = \{\exists\; s < t:\; X_s = 0\} $  and $\varrho^{(0)}_{n}(X) $ the statistic
\begin{equation}
\label{eq_first_stickiness_estimator}
\varrho^{(0)}_{n}(X) := 
\left( \int_{\IR} g \vd x \right) \frac{1}{\un} \frac{\sum_{i=1}^{[nt]} \indicb{\hfprocess{X}{}{i-1}=0}}{
	\sum_{i=1}^{[nt]} g(\un \hfprocess{X}{}{i-1}) }.
\end{equation}
Then, $\varrho^{(0)}_{n}(X) $ is a consistent estimator of $\rho $, conditionally on the event $\mathcal H_t $, i.e., $\varrho^{(0)}_{n}(X) \lra \rho $, in $\Prob^{\mathcal H_t}_x $-probability, as $n\lra \infty $.

We now devise a consistent stickiness parameter estimator based
on number of crossings statistics.
This provides an alternative estimator to~\eqref{eq_first_stickiness_estimator}.

Like~\eqref{eq_first_stickiness_estimator}, this
new estimator also converges conditionally on the event that $0$ has been reached before time $t$.
If the threshold is not reached, the number of crossings statistics are all trivially null.

Then, the following conditional convergence in probability holds.

\begin{proposition}
\label{prop_estimation}
%Let $X$ be the sticky Brownian motion of stickiness parameter $\rho>0 $,  defined on the filtered probability space $(\Omega,\bF, \process{\bF_t},\Prob_x) $ such that $\Prob_x$-\as, $X_0=x$,
Consider the setting of Theorem~\ref{thm_main} 
and let $t>0$.
The statistic 
\begin{equation}
	\label{eq_main_estimation}
	\begin{aligned}
		\varrho_n(X) &= \left(4 \sqrt{\frac 2 \pi} \right) \frac{1}{\sqn}\frac{\cC{2}{n,t}(X)}{ \cC{1}{n,t}(X)}
	\end{aligned}
\end{equation}
is a consistent estimator of $\rho $, conditionally on the event $\mathcal H_t $, i.e., $\varrho_n(X) \lra \rho $, in $\Prob^{\mathcal H_t}_x $-probability, as $n\lra \infty $.
\end{proposition}

\begin{remark}
The result holds also if  in \eqref{eq_main_estimation} one replaces $(\cC{1}{n,t}(X),\cC{2}{n,t}(X)) $ with $(\cB{1}{n,t}(X),\cB{2}{n,t}(X)) $ or $(\cD{1}{n,t}(X),\cD{2}{n,t}(X)) $.
\end{remark}

%%%%%%%%%%%%%%%%%%%%%%%%%%%%%%%%%%%%%%%%%%%%%%%%%%%%%%%%%%%%%%%%%%%%%%%%%%%%%%%%%%%%%%%
\section{Proofs of main results}
\label{sec_proof_B}
%%%%%%%%%%%%%%%%%%%%%%%%%%%%%%%%%%%%%%%%%%%%%%%%%%%%%%%%%%%%%%%%%%%%%%%%%%%%%%%%%%%%%%%

In this section we prove Theorem~\ref{thm_main} and Proposition~\ref{thm_main_B} for each type of crossing. The proofs are organized by crossing type: Section~\ref{ssec_type0} covers Type-0 crossings, Section~\ref{ssec_type1} covers Type-1 crossings, and Section~\ref{ssec_type2} covers Type-2 crossings. Following this, Section~\ref{sec_bouncings} contains the proof of Theorem~\ref{thm_main_R0} (asymptotics of bouncings statistics), and finally, Section~\ref{sec:proof_prop_estim} contains the proof of Proposition~\ref{prop_estimation} (stickiness parameter estimation).

For these, we need some results on the asymptotic behavior of the kernel, whose proof is deferred to Appendix~\ref{app_asymptotics}. In the statement we use the notation~\eqref{def_mn}--\eqref{notaz_m} of Proposition~\ref{prop_SMA_1}.

\begin{lemma}
\label{lem_transition_asymptotics}
%Let $X$ be the sticky Brownian motion of stickiness parameter $\rho>0 $, defined on the probability space $(\Omega,\process{\bF_t},\Prob_x) $ such that $X_0 = x $, $\Prob_x$-\as.
Consider the setting of Theorem~\ref{thm_main}. 
Let $ (f_n,k_n,g_n,h_n;\,n\in\IN) $ be the functions defined for all $n \in \IN$ and $x\in \IR $ by 
$f_n(x) := \Prob_{x} \left(  X_{1/n} = 0 \right) $, 
$k_n(x) := \indicb{x\not = 0} f_n(x/\sqn) $, 
$g_n(x) := \Prob_{x} \left( x X_{1/n} < 0 \right) $, 
$h_n(x) := \sqn g_n(x/\sqn) $.
It holds that 
\begin{enumerate}
	[label={\upshape (\roman*)}]
	\item \label{lem_fn0_convergence} $\lim_{n\ra \infty} f_n(0) = 1 $ and 
	$\lim_{n\ra \infty} \sqn (1- f_n(0)) = 
	{2 \sqrt{2}}/{\rho \sqrt \pi} $,
	\item \label{lem_kn_convergence} the sequence $(k_n)_n $ satisfies~\eqref{eq_connd_gn_aggreg} and $\lim_{n\ra \infty} m_{\sqn \rho}( k_n )  =  2 \sqrt{{2}/{\pi}} $,
	\item \label{lem_hn_convergence} the sequence $(h_n)_n $ satisfies~\eqref{eq_connd_gn_aggreg} and $\lim_{n\ra \infty} m_{\sqn \rho}( h_n )  = 1/{\rho} $.
\end{enumerate}
\end{lemma}

\subsection{Crossings of type $0$}
\label{ssec_type0}

Here, we prove all results regarding number of Type-0 crossings.
Namely, Theorem~\ref{thm_main}\ref{item_strict} and Proposition~\ref{thm_main_B}\ref{item_strict_B}.
We first prove Proposition~\ref{thm_main_B}\ref{item_strict_B} relying on the sticky kernel asymptotics from Lemma~\ref{lem_transition_asymptotics} and the convergence result~Proposition~\ref{prop_SMA_1}.
Then, using Lemma~\ref{lem_Genon_Catalot} (conditional and unconditioned versions of the statistic have the same asymptotic) we prove Theorem~\ref{thm_main}\ref{item_strict} by reducing to Proposition~\ref{thm_main_B}\ref{item_strict_B}. 

\begin{proof}
[Proof of Proposition~\ref{thm_main_B}\ref{item_strict_B}]
We begin by expressing the statistic $\Cond{0}{n,t}(X)$ in terms of the functions $(g_n)$ and $(h_n)$ from Lemma~\ref{lem_transition_asymptotics}. 
Indeed, by the Markov property it holds that
\begin{equation}
	\Cond{0}{n,t}(X) = \sum_{i=1}^{[nt]} \Esp_x \bigbraces{\indicb{\hfprocess{X}{}{i-1}
			\hfprocess{X}{}{i}< 0}\big| \bF_{\frac{i-1}{n}}}
	= \sum_{i=1}^{[nt]} g_n(\hfprocess{X}{}{i-1})
	= \frac{1}{\sqn} \sum_{i=1}^{[nt]} h_n(\sqn \hfprocess{X}{}{i-1}).
\end{equation}
By Proposition~\ref{prop_SMA_1} and Lemma~\ref{lem_transition_asymptotics}\ref{lem_hn_convergence}, we have
\begin{equation}
	\Cond{0}{n,t}(X) \convergence{\Prob_x} \frac{1}{\rho} \loct{X}{0}{t}
	= \frac{1}{\rho^{2}} \cloct{\mathcal O}{X}{0}{t},
\end{equation}
which completes the proof.
\end{proof}

\begin{proof}
[Proof of Theorem~\ref{thm_main}\ref{item_strict}]
We consider the families of $\sigma$-algebras $(\bF^{n}_{i})_{i,n} $
and random variables $(\chi^{n}_{i})_{i,n} $,
defined for all $i,n $ by 
\begin{equation}
	\begin{aligned}
		\bF^{n}_{i} &:= \bF_{i/n},
		&\text{and}& 
		&\chi^{n}_{i} &:= \frac{1}{\un} \indicb{
			\hfprocess{X}{}{i-1}\hfprocess{X}{}{i}< 0}
	\end{aligned}
\end{equation}
so that
\begin{equation}
	\begin{aligned}
		\frac{1}{\un}\cC{0}{n,t}(X) &= \sum_{i=1}^{[nt]} \chi^{n}_{i} ,   &\text{and}  &
		&\frac{1}{\un}\Cond{0}{n,t}(X) &= \sum_{i=1}^{[nt]} \Esp_x \left( \chi^{n}_{i} \big| \bF^{n}_{i-1} \right).
	\end{aligned}
\end{equation}
From Proposition~\ref{thm_main_B}\ref{item_strict_B}, 
\begin{equation}
	\sum_{i=1}^{[nt]} \Esp_x \left( \chi^{n}_{i} \big| \bF^{n}_{i-1} \right)
	= \frac{1}{\un} \Cond{0}{n,t}(X)
	\convergence{\Prob_x} 0
\end{equation}
and, since $(\chi^n_i)^2= \chi^n_i/\un$,
\begin{equation}
	\sum_{i=1}^{[nt]} \Esp_x \left( \left(\chi^{n}_{i}\right)^{2} \big| \bF^{n}_{i-1} \right)
	= \frac{1}{\un} \sum_{i=1}^{[nt]} \Esp_x \left( \chi^{n}_{i} \big| \bF^{n}_{i-1} \right)
	= \frac{1}{\un^{2}} \Cond{0}{n,t}(X)
	\convergence{\Prob_x} 0.
\end{equation}
Thus, from Lemma~\ref{lem_Genon_Catalot},
\begin{equation}
	\frac{1}{\un}\cC{0}{n,t}(X)
	\convergence{\Prob_x} 0.
\end{equation}
From Lemma~\ref{lem_ucp_convergence_condition}, the convergence is also uniform in time, in probability (ucp). 
This finishes the proof.	
\end{proof}

\subsection{Crossings of type $1$}
\label{ssec_type1}

Here, we prove all results regarding number of Type-1 crossings. Namely, Theorem~\ref{thm_main}\ref{item_large} and part of  Proposition~\ref{thm_main_B}\ref{item_large}.

The proof relies on the study of the limit behavior of the following quantity which 
according to our naming conventions is referred to as the conditional difference statistic:
\begin{equation}
\begin{aligned}
	\CondD{1}{n,t}(X) &:= \sum_{i=1}^{[nt]} \Esp_x \Bigbraces{\indicb{\hfprocess{X}{}{i-1}=0,\,\hfprocess{X}{}{i}\not= 0}
		+ \indicb{\hfprocess{X}{}{i-1}\not =0,\,\hfprocess{X}{}{i}= 0}	
		\big| \bF_{\frac{i-1}{n}}},
\end{aligned}
\end{equation}
which satisfies $\Cond{1}{n,t}(X) = \CondD{1}{n,t}(X) + \Cond{0}{n,t}(X) $, involving the Type-1 and Type-0 conditional statistics.  

We first establish the limit behavior of $\CondD{1}{n,t}(X)$ using Lemmas~\ref{thm_Altmeyer} (occupation time approximation) and~\ref{lem_transition_asymptotics} (sticky kernel asymptotics). 
The asymptotic of the Type-1 conditional crossing statistic $\Cond{1}{n,t}(X) $ follows from the decomposition 
$\Cond{1}{n,t}(X) = \CondD{1}{n,t}(X) + \Cond{0}{n,t}(X) $ and Proposition~\ref{thm_main_B}\ref{item_strict_B} (asymptotic of the respective Type-0 statistic).
Finally, by Lemma~\ref{lem_Genon_Catalot} we deduce Theorem~\ref{thm_main}\ref{item_large} from its conditional version.

\begin{lemma}
\label{lem_crossings_exp_1}
The following convergences hold: 
\begin{equation}
	\frac{1}{\sqn}\CondD{1}{n,t}(X)
	\convergence{\Prob_x}  4 \sqrt{\frac 2 \pi} \loct{X}{0}{t}.
\end{equation}
\end{lemma}

\begin{proof}
Regarding the first term of $\CondD{1}{}$, by Lemmas~\ref{thm_Altmeyer} and~\ref{lem_transition_asymptotics}\ref{lem_fn0_convergence} we have that
\begin{multline}
	\label{eq_proof_Z1_eq1}
	\frac{1}{\sqrt{n}}\sum_{i=1}^{[nt]} \Esp_x \bigsqbraces{\indicb{\hfprocess{X}{}{i-1}=0;\;
			\hfprocess{X}{}{i}\not = 0} \big| \bF_{\frac{i-1}{n}}}
	= \frac{1}{\sqn} \Prob_{0} \left( \hfprocess{X}{}{1}\not = 0 \right) \sum_{i=1}^{[nt]} \indicb{\hfprocess{X}{}{i-1}=0}
	\\= \sqn \left(1 - f_n( 0)\right) \frac{1}{n} \sum_{i=1}^{[nt]} \indicb{\hfprocess{X}{}{i-1}=0}
	\convergence{\Prob_x} 
	\frac{ 2 \sqrt{2}}{\rho \sqrt \pi}  \cloct{\mathcal O}{X}{0}{t}
	%		= \frac{ 2 \sqrt{2}}{\rho \sqrt \pi} \rho \loct{X}{0}{t}
	= 2 \sqrt{\frac 2 \pi} \loct{X}{0}{t}.
\end{multline}

Regarding the second term of $\CondD{1}{}$, by the Markov property we have that	\begin{equation}
	\begin{aligned}
		\frac{1}{\sqrt{n}}\sum_{i=1}^{[nt]} \Esp_x \bigsqbraces{\indicb{\hfprocess{X}{}{i-1}\not =0;\;
				\hfprocess{X}{}{i} = 0} \big| \bF_{\frac{i-1}{n}}}
		&= 
		\frac{1}{\sqrt{n}}\sum_{i=1}^{[nt]} \indicb{\hfprocess{X}{}{i-1}\not =0} \Prob_{\hfprocess{X}{}{i-1}} 
		\left( \hfprocess{X}{}{1} = 0 \right)
		\\ &= \frac{1}{\sqrt{n}}\sum_{i=1}^{[nt]} k_n(\sqn \hfprocess{X}{}{i-1}),
	\end{aligned}
\end{equation}
where $(k_n)_n $ is the sequence of functions defined in Lemma~\ref{lem_transition_asymptotics}.
From Proposition~\ref{prop_SMA_1} and Lemma~\ref{lem_transition_asymptotics}\ref{lem_kn_convergence}, 
\begin{equation}
	\label{eq_proof_Z1_eq2}
	\frac{1}{\sqrt{n}}\sum_{i=1}^{[nt]} k_n(\sqn \hfprocess{X}{}{i-1})
	\convergence{\Prob_x}
	\left( \lim_{n\ra \infty} m_{\sqn \rho} (k_n) \right) \loct{X}{0}{t}
	=  2 \sqrt{\frac{2}{\pi}} \loct{X}{0}{t}.
\end{equation}
Combining the relations~\eqref{eq_proof_Z1_eq1} and~\eqref{eq_proof_Z1_eq2}
yields the desired result. 
\end{proof}

\begin{proof}
[Proof of Proposition~\ref{thm_main_B}\ref{item_large_B} for $j=1 $]
Essentially, it suffices to observe that
\begin{equation}
	\frac{1}{\sqn} \Cond{1}{n,t}(X) = 
	\frac{1}{\sqn}\CondD{1}{n,t}(X)
	+ \frac{1}{\sqn} \Cond{0}{n,t}(X).
\end{equation}
Then, by Lemma~\ref{lem_crossings_exp_1} and Proposition~\ref{thm_main_B}\ref{item_strict_B},
\begin{equation}
	\frac{1}{\sqn} \Cond{1}{n,t}(X) \convergence{\Prob_x}
	4 \sqrt{\frac{2}{\pi}} \loct{X}{0}{t}.
\end{equation}
By Lemma~\ref{lem_ucp_convergence_condition}, the convergence is also uniform in time, in probability (ucp).
\end{proof}

\begin{proof}
[Proof of Theorem~\ref{thm_main}\ref{item_large}]
The proofs works by first proving convergence of the Type-1 zero statistic $\cD{1}{\cdot,t}  $ by reduction to its conditional version $\CondD{1}{\cdot,t} $. 
We then conclude from previous result by re-writting $\cC{1}{\cdot,t} $ as $\cD{1}{\cdot,t} + \cC{0}{\cdot,t} $.

For the reduction argument, we consider the families of $\sigma$-algebras $(\bF^{n}_{i})_{i,n} $
and random variables $(\chi^{n}_{i})_{i,n} $,
defined for all $i\in \IN, n >0$ by 
\begin{equation}
	\begin{aligned}
		\bF^{n}_{i} &:= \bF_{i/n},
		& \text{and}&
		&\chi^{n}_{i} &:= \frac{1}{\sqn} \Bigbraces{ \indicb{\hfprocess{X}{}{i-1}\hfprocess{X}{}{i}= 0}
			- \indicb{\hfprocess{X}{}{i-1}=\hfprocess{X}{}{i} = 0}}, 
	\end{aligned}
\end{equation}
so that
\begin{equation}
	\begin{aligned}
		\frac{1}{\sqn}\cD{1}{n,t}(X)&=
		\sum_{i=1}^{[nt]}\chi^{n}_{i} 
		& \text{and} &
		& \frac{1}{\sqn}\CondD{1}{n,t}(X)&= \sum_{i=1}^{[nt]}\Esp_x \left( \chi^{n}_{i} \big| \bF^{n}_{i-1} \right).
	\end{aligned}
\end{equation}

From Lemma~\ref{lem_crossings_exp_1}, %Proposition~\ref{thm_main_B}.\ref{item_large_B}, 
\begin{equation}
	\sum_{i=1}^{[nt]} \Esp_x \left( \chi^{n}_{i} \big| \bF^{n}_{i-1} \right)
	%= \frac{1}{\sqn} \Cond{1}{n,t}(X)
	\convergence{\Prob_x} 4 \sqrt{\frac{2}{\pi}} \loct{X}{0}{t}
\end{equation}
and, since $(\chi^n_i)^2= \chi^n_i/\sqn$,
\begin{equation}
	\sum_{i=1}^{[nt]} \Esp_x \left( \left(\chi^{n}_{i}\right)^{2} \big| \bF^{n}_{i-1} \right)
	%= \frac{1}{n} \Cond{1}{n,t}(X) 
	\convergence{\Prob_x} 0.
\end{equation}
By Lemma~\ref{lem_Genon_Catalot} we therefore have that
\begin{equation}
	\frac{1}{\sqn}\cD{1}{n,t}(X) \convergence{\Prob_x} 
	4 \sqrt{\frac{2}{\pi}} \loct{X}{0}{t}.
\end{equation}
Combining this with Theorem~\ref{thm_main}\ref{item_strict}
ensures that 
\begin{equation}
	\frac{1}{\sqn}\cC{1}{n,t}(X) 
	= \frac{1}{\sqn}\cD{1}{n,t}(X) + \frac{1}{\sqn}\cC{0}{n,t}(X)
	\convergence{\Prob_x}  4 \sqrt{\frac{2}{\pi}} \loct{X}{0}{t}.
\end{equation}
From Lemma~\ref{lem_ucp_convergence_condition}, the convergence is also uniform in time, in probability (ucp). 
This finishes the proof.	
\end{proof}

\subsection{Crossings of type $2$}
\label{ssec_type2}

Here, we prove all results regarding the number of Type-2 crossings, completing the proofs of Theorem~\ref{thm_main} and Proposition~\ref{thm_main_B}. The proofs consist of combining results on Type-1 crossings with the occupation time approximation (Lemma~\ref{thm_Altmeyer}).

Specifically, we show that both the Type-2 crossing statistic and its conditional version are asymptotically equivalent to the occupation time approximation from  Lemma~\ref{thm_Altmeyer}.

\begin{proof}
[Proof of Theorem~\ref{thm_main}\ref{item_stays}]
Recall the decomposition 
\begin{equation}
	\cC{2}{n,t}(X) =  \cC{1}{n,t}(X) + \sum_{i=1}^{[nt]} \indicb{\hfprocess{X}{}{i-1} =0, \hfprocess{X}{}{i} = 0},
\end{equation}	
where we can rewrite the second term as
\begin{equation}
	\indicb{\hfprocess{X}{}{i-1} =0, \hfprocess{X}{}{i} = 0} 
	= \indicb{\hfprocess{X}{}{i-1} =0} -\indicb{\hfprocess{X}{}{i-1} =0, \hfprocess{X}{}{i}\neq 0}.
\end{equation}
Since $ \{\hfprocess{X}{}{i-1} =0, \hfprocess{X}{}{i}\neq 0\}$ corresponds to a particular case of type 1 crossing, we have 
$0 \le \cC{1}{n,t}(X) - \sum_{i=1}^{[nt]} \indicb{\hfprocess{X}{}{i-1} =0, \hfprocess{X}{}{i}\neq 0} \le \cC{1}{n,t}(X) $.
This yields the bound
\begin{equation}
	\frac1n \sum_{i=1}^{[nt]} \indicb{\hfprocess{X}{}{i-1}=0}
	\,\le\, 
	\frac{1}{n}\cC{2}{n,t}(X)
	\,\le\,  
	\frac1n \sum_{i=1}^{[nt]} \indicb{\hfprocess{X}{}{i-1}=0}
	+ \frac1n \cC{1}{n,t}(X)
\end{equation}
By Item~\ref{item_large} and Lemma~\ref{thm_Altmeyer}, both bounds converge to $\rho \loct{X}{0}{t}$. Therefore, by the squeeze theorem:
\begin{equation}
	\forall t\ge 0:
	\qquad
	\frac{1}{n}\cC{2}{n,t}(X) \convergence{\Prob_x} 
	\int_{0}^{t} \indicb{X_s = 0} \vd s
	= \rho \loct{X}{0}{t}.
\end{equation}
Lemma~\ref{lem_ucp_convergence_condition} ensures ucp~convergence, which completes the proof. 
\end{proof}

\begin{proof}
[Proof of Proposition~\ref{thm_main_B}\ref{item_large_B} for $j=2 $] 
Observe that
\begin{equation}
	\begin{aligned}
		\frac{1}{n}\Cond{2}{n,t}(X) 
		&= %\frac{1}{n} \CondD{2}{n,t}(X) + \frac{1}{n}\Cond{1}{n,t}(X) =
		\frac{1}{n} \sum_{i=1}^{[nt]} \Esp_x \Bigbraces{\indicb{\hfprocess{X}{}{i-1}=\hfprocess{X}{}{i}= 0} \big| \bF_{\frac{i-1}{n}}} 
		+ \frac{1}{n}\Cond{1}{n,t}(X).
	\end{aligned}
\end{equation}
With the same argument as in the proof of Theorem~\ref{thm_main}\ref{item_stays}, 
\begin{equation}
	\frac{1}{n}\sum_{i=1}^{[nt]} \indicb{\hfprocess{X}{}{i-1} = 0} 
	\le
	\frac{1}{n}\Cond{2}{n,t}(X) \le 
	\frac{1}{n}\sum_{i=1}^{[nt]} \indicb{\hfprocess{X}{}{i-1} = 0} 
	+ \frac{1}{n}\Cond{1}{n,t}(X)
\end{equation}
From Item~\ref{item_large} for $j=1 $ and Lemma~\ref{thm_Altmeyer}, we have that
\begin{equation}
	\forall t\ge 0:
	\qquad
	\frac{1}{n}\Cond{2}{n,t}(X) \convergence{\Prob_x} 
	\int_{0}^{t} \indicb{X_s = 0} \vd s
	= \rho \loct{X}{0}{t}.
\end{equation}
By Lemma~\ref{lem_ucp_convergence_condition}, the convergence is ucp. 
This finishes the proof.	
\end{proof}

%%%%%%%%%%%%%%%%%%%%%%%%%%%%%%%%%%%%%%%%%%%%%%%%%%%%%%%%%%%%%%%%%%%%%%%%%%%%%%%%%%%%%%%
\subsection{Proof of Theorem~\ref{thm_main_R0}}
\label{sec_bouncings}
%%%%%%%%%%%%%%%%%%%%%%%%%%%%%%%%%%%%%%%%%%%%%%%%%%%%%%%%%%%%%%%%%%%%%%%%%%%%%%%%%%%%%%%

The proof relies essentially in linking bouncing and crossing behaviors of type 0 via a reflection principle at $0$ for the sticky Brownian motion.
Then, analogous to the results for Type-1 and Type-2 crossings, the leading terms in the number of Type-1 and Type-2 bouncings ($\cB{1}{} $, $\cB{2}{} $) are given by the corresponding difference statistics ($\cD{1}{} $, $\cD{2}{} $).

For the statement and the proof of the reflection principle, let $\tau^{Y}_{0} $ to be the hitting time of $0$ by the process $Y$, defined as
\begin{equation}
\label{eq_txt_hitting_times}
\tau^{Y}_{0} = \inf \{t>0:\; Y_t = 0\}.
\end{equation}
We are now ready to state this intermediary result. The proof is deferred to Appendix~\ref{app_reflection_sticky}.

\begin{lemma}[Reflection principle for sticky Brownian motion]
\label{lem_sticky_refl_principle}
%Consider the setting of Theorem~\ref{thm_main}. 
Let $X$ be the sticky Brownian motion of stickiness parameter $\rho>0$ defined on the filtered probability space $ \mathcal P_x = (\Omega, \bF, \process{\bF_t},\Prob_x) $ such that $\Prob_x$-\as, $X_0 = x$ (in particular, $X$ is $\process{\mathcal F_t}$-adapted). 
Let $X' $ be the process defined for all $t\ge 0 $ by
\begin{equation}
	%		\label{eq_lem_sticky_refl_at_0}
	X'_t = 
	\begin{dcases}
		X_t & t< \tau^{X}_{0},\\
		-X_t & t\ge \tau^{X}_{0}.\\
	\end{dcases}
\end{equation}
The process $X'$ defined this way is a sticky Brownian motion of stickiness parameter $\rho$.
\end{lemma}

We are now ready to address the proof of Theorem~\ref{thm_main_R0}.

\begin{proof}[Proof of Theorem~\ref{thm_main_R0}\ref{item_strict_R0}]
In this proof, let $ \CondB{0}{}$ be the conditional version of $ \cB{0}{}$, defined for all $n,t >0$ by 
\begin{equation}
	\CondB{0}{n,t}(X) = 
	\sum_{i=1}^{[nt]} 
	\Esp_x \Bigbraces{\indic{U^{n}_{i}(X)} \indicb{\hfprocess{X}{}{i-1}\hfprocess{X}{}{i}>0}
		\big| \bF_{\frac{i-1}{n}}}.
\end{equation}
We also introduce the notation of Markovian families used in~\cite[Section~2.1]{Fre}.
Let $X'$ be a diffusion on $J$, an interval of $\IR $,
defined on the probability space $\mathcal P_x$ such that 
$\Prob_x $-\as, $X'_0 = x $. 
We denote by $(Q^{X'}_y,\; y\in J )$ the canonical diffusion  
on the path-space $C([0,\infty),\IR)$ and $Y $ be the coordinate process. 
Then, for every almost surely finite stopping time $\tau$, we have
\begin{equation}
	\label{eq_Markovian_family_formalism}
	Q^{X'}_{X'_{\tau}} \bigbraces{ h(Y_t) } = \Esp_x \bigbraces{ h(X'_{\tau+t}) \big| \bF_\tau }.
\end{equation}
From the strong-Markov property, using the above notation and~\eqref{eq_txt_hitting_times}, we have 
\begin{align}
	\Esp_x \Bigbraces{\indic{U^{n}_{i}(X)} \indicb{\hfprocess{X}{}{i-1}\hfprocess{X}{}{i}>0}
		\big| \bF_{\frac{i-1}{n}}}
	={} & Q^{X}_{\hfprocess{X}{}{i-1}} \Bigbraces{\indicb{\tau^{Y}_0<
			\frac 1 n}
		\indicb{Y_0 \hfprocess{Y}{}{1}>0}}
	\\
	={} &
	\indicb{\hfprocess{X}{}{i-1}>0} Q^{X}_{\hfprocess{X}{}{i-1}} \Bigbraces{\indicb{\tau^{Y}_0<
			\frac 1 n}
		\indicb{\hfprocess{Y}{}{1}>0}}
	\\ &+
	\indicb{\hfprocess{X}{}{i-1}<0} Q^{X}_{\hfprocess{X}{}{i-1}} \Bigbraces{\indicb{\tau^{Y}_0< \frac 1 n}
		\indicb{\hfprocess{Y}{}{1}<0}},
\end{align}
and 
\begin{equation}
	\begin{aligned}
		& \Esp_x \biggbraces{ \indicb{\hfprocess{X}{}{i-1}\hfprocess{X}{}{i}<0}
			\Big| \bF_{\frac{i-1}{n}}}
		= Q^{X}_{\hfprocess{X}{}{i-1}} \Bigbraces{
			\indicb{Y_0 \hfprocess{Y}{}{1}<0}}
		\\
		& \qquad \qquad = \indicb{\hfprocess{X}{}{i-1}<0}  Q^{X}_{\hfprocess{X}{}{i-1}} \Bigbraces{
			\indicb{\hfprocess{Y}{}{1}>0}}
		+ \indicb{\hfprocess{X}{}{i-1}>0}  Q^{X}_{\hfprocess{X}{}{i-1}} \Bigbraces{
			\indicb{\hfprocess{Y}{}{1}<0}}.
	\end{aligned}
\end{equation}		
From Lemma~\ref{lem_sticky_refl_principle}, we have
\begin{equation}
	\begin{aligned}
		\indicb{\hfprocess{X}{}{i-1}>0} Q^{X}_{\hfprocess{X}{}{i-1}} \Bigbraces{\indicb{\tau^{Y}_0<
				\frac 1 n}
			\indicb{\hfprocess{Y}{}{1}>0}} 
		&= \indicb{\hfprocess{X}{}{i-1}>0} Q^{X}_{\hfprocess{X}{}{i-1}} \Bigbraces{\indicb{\tau^{Y}_0<
				\frac 1 n}
			\indicb{\hfprocess{Y}{}{1}<0}}
		\\&=  \indicb{\hfprocess{X}{}{i-1}>0} Q^{X}_{\hfprocess{X}{}{i-1}} \Bigbraces{\indicb{\hfprocess{Y}{}{1}<0}}
	\end{aligned}
\end{equation}
and
\begin{equation}
	\begin{aligned}
		\indicb{\hfprocess{X}{}{i-1}<0} Q^{X}_{\hfprocess{X}{}{i-1}} \Bigbraces{\indicb{\tau^{Y}_0<
				\frac 1 n}
			\indicb{\hfprocess{Y}{}{1}<0}} 
		&= \indicb{\hfprocess{X}{}{i-1}<0} Q^{X}_{\hfprocess{X}{}{i-1}} \Bigbraces{\indicb{\tau^{Y}_0<
				\frac 1 n}
			\indicb{\hfprocess{Y}{}{1}>0}}
		\\&=  \indicb{\hfprocess{X}{}{i-1}<0} Q^{X}_{\hfprocess{X}{}{i-1}} \Bigbraces{\indicb{\hfprocess{Y}{}{1}>0}}.
	\end{aligned}
\end{equation}
Therefore, we obtain
\begin{equation}
	\Esp_x \Bigbraces{\indic{U^{n}_{i}(X)} \indicb{\hfprocess{X}{}{i-1}\hfprocess{X}{}{i}>0}
		\big| \bF_{\frac{i-1}{n}}}
	= \Esp_x \biggbraces{ \indicb{\hfprocess{X}{}{i-1}\hfprocess{X}{}{i}<0}
		\Big| \bF_{\frac{i-1}{n}}}
\end{equation}	
and thus
\begin{equation}
	\label{eq_proof_B0_T0_relation}
	\CondB{0}{n,t} = \Cond{0}{n,t}. 
\end{equation}
Let us consider the families of $\sigma$-algebras $(\bF^{n}_{i})_{i,n} $
and random variables $(\chi^{n}_{i})_{i,n} $,
defined for all $i,n $ by 
\begin{equation}
	\begin{aligned}
		\bF^{n}_{i} &= \bF_{i/n}
		&\quad \text{and}\quad&
		&\chi^{n}_{i} &= \frac{1}{\un} \indic{U^{n}_{i}(X)} \indicb{
			\hfprocess{X}{}{i-1}\hfprocess{X}{}{i}> 0}.
	\end{aligned}			
\end{equation}
Then, 
\begin{equation}
	\begin{aligned}
		\frac{1}{\un}\cB{0}{n,t}(X) &= \sum_{i=1}^{[nt]} \chi^{n}_{i} &
		&\quad \text{and}\quad & 
		\frac{1}{\un}\CondB{0}{n,t}(X) &= \sum_{i=1}^{[nt]} \Esp_x \left( \chi^{n}_{i} \big| \bF^{n}_{i-1} \right).
	\end{aligned}			
\end{equation}
From Proposition~\ref{thm_main_B}\ref{item_strict_B} and \eqref{eq_proof_B0_T0_relation}, 
\begin{equation}
	\sum_{i=1}^{[nt]} \Esp_x \left( \chi^{n}_{i} \big| \bF^{n}_{i-1} \right)
	= \frac{1}{\un} \CondB{0}{n,t}(X) = \frac{1}{\un} \Cond{0}{n,t}(X)
	\convergence{\Prob_x} 0
\end{equation}
and, since $(\chi^{n}_i)^{2}=\chi^{n}_i/\un $,
\begin{align}
	\sum_{i=1}^{[nt]} \Esp_x \left( \left(\chi^{n}_{i}\right)^{2} \big| \bF^{n}_{i-1} \right)
	&= \frac{1}{\un} \sum_{i=1}^{[nt]} \Esp_x \left( \chi^{n}_{i} \big| \bF^{n}_{i-1} \right)
	\\ &= \frac{1}{\un^{2}} \CondB{0}{n,t}(X) = \frac{1}{\un^{2}} \Cond{0}{n,t}(X)
	\convergence{\Prob_x} 0.
\end{align}
Thus, from Lemma~\ref{lem_Genon_Catalot},
\begin{equation}
	\frac{1}{\un}\cB{0}{n,t}(X) = \sum_{i=1}^{[nt]} \chi^{n}_{i} 
	\convergence{\Prob_x} 0.
\end{equation}
From Lemma~\ref{lem_ucp_convergence_condition}, the convergence is also uniform in time, in probability (ucp). 
This finishes the proof of~Item~\ref{item_strict_R0}.	
\end{proof}

\begin{proof}[Proof of Theorem~\ref{thm_main_R0}\ref{item_large_R0}]
We observe that, for $j=1,2$,
$	\cD{j}{n,\cdot}(X) = \cB{j}{n,\cdot}(X) - \cB{j-1}{n,\cdot}(X).$ 
Thus, from Theorem~\ref{thm_main}, 
%Corollary~\ref{rmk_wtC_statistics} 
and Item~\ref{item_strict_R0}, it holds that
\begin{equation}
	\begin{aligned}
		\frac{1}{\sqn}\cB{1}{n,\cdot}(X) &\convergence{\Prob_x}
		4 \sqrt{\frac{2}{\pi}} \loct{X}{0}{}
		& \text{and}&
		& \frac{1}{n}\cB{2}{n,\cdot}(X) &\convergence{\Prob_x}
		\rho \loct{X}{0}{t}.
	\end{aligned}
\end{equation}
From Lemma~\ref{lem_ucp_convergence_condition}, the convergence is uniform in time, in $\Prob_x$-probability ($\Prob_x$-ucp).
This completes the proof.
\end{proof}

%%%%%%%%%%%%%%%%%%%%%%%%%%%%%%%%%%%%%%%%%%%%%%
\subsection{Proof of Proposition~\ref{prop_estimation}} \label{sec:proof_prop_estim}

\begin{proof}
From Theorem~\ref{thm_main}, $\cC{2}{n,t}(X) \convergence{\Prob_x} \rho \loct{X}{0}{t}  $
and $\cC{1}{n,t}(X) \convergence{\Prob_x}  4 \sqrt{\frac 2 \pi} \loct{X}{0}{t}  $.
This and the fact that $\loct{X}{0}{t}=\loct{X}{0}{t} \indicb{\loct{X}{0}{t}>0}$ ensure that
$\cC{2}{n,t}(X) \indicb{\loct{X}{0}{t}>0} \convergence{\Prob_x} \rho \loct{X}{0}{t} $
and $\cC{1}{n,t}(X) \indicb{\loct{X}{0}{t}>0}\convergence{\Prob_x}  4 \sqrt{\frac 2 \pi} \loct{X}{0}{t}  $, 
and 
\begin{equation}
	\indicb{\loct{X}{0}{t}>0} \varrho_n(X)
	= \indicb{\loct{X}{0}{t}>0} \left( 4 \sqrt{\frac 2 \pi } \right) \frac{1}{\sqn}\frac{\cC{2}{n,t}(X)}{\cC{1}{n,t}(X)}
	\convergence{\Prob_x} \rho \indicb{\loct{X}{0}{t}>0} .
\end{equation}
Note that $\Prob_x$-\as, $ \{\loct{X}{0}{t}>0\} = \mathcal H_t $~(this follows, e.g., from~\cite[Corollary~29.18]{Kal}) %~\eqref{eq_txt_support_loctime}) 
and for every $\varepsilon >0$
\begin{equation} 
	\Prob_x(\mathcal H_t \cap \{|\varrho_n -\rho|>\varepsilon\} \}) 
	=   
	\Prob_x(\{\loct{X}{0}{t}>0\} \cap \{|\varrho_n -\rho|>\varepsilon\} \})
	= 
	\Prob_x(|\varrho_n -\rho| \indicb{\loct{X}{0}{t}>0}>\varepsilon)
\end{equation}
which converges to 0 as $n\to \infty$, as we just showed.
Since $\Prob_x(\mathcal H_t)>0 $, from Bayes rule,
we prove the desired convergence.
\end{proof}

%%%%%%%%%%%%%%%%%%%%%%%%%%%%%%%%%%%%%%%%%%%%%%%%%%%%%%%%%%%%%%%%%%%%%%%%%%%%%%%%%%%%%%%
\section{Extensions}
\label{sec_SID}
%%%%%%%%%%%%%%%%%%%%%%%%%%%%%%%%%%%%%%%%%%%%%%%%%%%%%%%%%%%%%%%%%%%%%%%%%%%%%%%%%%%%%%%

In this section we extend the results on the sticky Brownian motion
to the sticky-reflected Brownian motion and to sticky Itô diffusions, i.e., processes
that solve a homogeneous SDE away from a sticky point.

\subsection{Bouncings of the sticky-reflected Brownian motion} \label{ssec_stickyrefl}

The sticky-reflected Brownian motion, also known as \textit{slowly reflected} Brownian motion, was
first discovered by Feller in his attempt to describe all possible ways to define a Brownian motion
on the positive semi-axis $[0,\infty) $,  see~\cite{Fel52}.
In particular, he discovered that 
the operator $\Lop = \D^{2}_x $ with domain
\begin{equation}
\dom(\Lop) = \left\{ f \in C_0([0,\infty))\cap C^{2}((0,\infty)):\; f'(0+)= \bigbraces{\rho/2} f''(0+) \right\},
\end{equation}
with $\rho>0$,
is the infinitesimal generator of a diffusion that spends a positive amount of time at $0$ and that behaves like a Brownian motion away from $0$.
Replacing the boundary condition $f'(0+)= \bigbraces{\rho/2} f''(0+)$ by the lateral condition $ f(0+)-f(0-) = \bigbraces{\rho/2} f''(0+) $ 
defines the sticky Brownian motion, or two-sided sticky Brownian motion, introduced in Section~\ref{ssec_stickyBM}.
A nice historical overview can be found in~\cite{Pesk2014}.

For simplicity, we define the sticky-reflected Brownian motion of stickiness parameter $\rho>0 $ as the unique (in law) 
weak solution to the system 
\begin{equation}
\label{eq_SDE_X}
\begin{dcases}
	\vd X_t = \indicb{X_t > 0} \vd B_t + \frac{1}{2} \vd \loct{X}{0}{t},\\
	\indicb{X_t = 0} \vd t = \frac{\rho}{2} \vd \loct{X}{0}{t},
\end{dcases}
\end{equation}
where $B$ is a standard Brownian motion (see~\cite[Theorem~5]{EngPes}). 

Let us note that the absolute value of a sticky Brownian motion is a sticky-reflected Brownian motion with the same stickiness parameter. More precisely, if $Y $ is a sticky Brownian motion of stickiness parameter $\rho $,
from the Tanaka formula~\cite[theorem~VI.1.2]{RevYor} and the SDE characterization~\ref{item_sticky_P2} for sticky Brownian motion, there exists a Brownian motion $B$ on an extension of the probability space such that  
\begin{equation}
\label{eq_proof_Y_SDE_part}
\begin{aligned}
	\vd |Y_t| &= \sgn(Y_t) \vd Y_t + \vd \loct{Y}{0}{t}\\
	&= \sgn(Y_t) \indicb{|Y_t|\not = 0} \vd B_t + \vd \loct{Y}{0}{t}.
\end{aligned}
\end{equation}
This, along with~\ref{item_sticky_P2}, yield that $\vd \qv{Y}_t = \vd \qv{|Y|}_t  = \indicb{Y_t \not = 0} \vd t$ and that, since $Y$ is a martingale, from~\cite[Corollary~VI.1.9]{RevYor},
\begin{equation}
\label{eq_proof_Y_loct_part}
\begin{aligned}
	\loct{|Y|}{0}{t} &= \lim_{\epsilon \ra 0} \frac{1}{\epsilon}
	\int_{0}^{t} \indicb{0<|Y_s|<\epsilon} \vd s \\
	&= \lim_{\epsilon \ra 0} \frac{1}{\epsilon}
	\int_{0}^{t} \indicb{0<Y_s<\epsilon} \vd s 
	+ \lim_{\epsilon \ra 0} \frac{1}{\epsilon}
	\int_{0}^{t} \indicb{0<-Y_s<\epsilon} \vd s 
	\\&= 2 \loct{Y}{0}{t} = 2\rho \int_{0}^{t}\indicb{Y_s = 0} \vd s.
\end{aligned}
\end{equation}
From~\eqref{eq_proof_Y_SDE_part} and~\eqref{eq_proof_Y_loct_part}, if $B'=\int_{0}^{\cdot} \sgn(Y_s) \vd B_s $, the pair $(|Y|, B') $
solves~\eqref{eq_SDE_X}.

We are now ready to state our main results regarding the number of bouncings of this process.
The proof reduces to the case of non-reflected sticky Brownian motion.

\begin{theorem}
\label{thm_main_R}
Let $X$ be the sticky-reflected Brownian motion of stickiness parameter $\rho>0 $, defined on the filtered probability space $(\Omega,\bF,\process{\bF_t},\Prob_x) $ such that $\Prob_x$-\as: $X_0 = x \geq 0$ (in particular, $X$ is $\process{\bF_t}$-adapted). 
Let also $\loct{X}{0}{} $ be the right local time at $0$ of $X$ 
and $\cB{0}{},\cB{1}{},\cB{2}{},\cD{1}{}, \cD{2}{} $ be the quantities defined in Section~\ref{sec_intro}
for the process $X$.
Then, the following convergences holds, 
\begin{enumerate}
	[label={ (\roman*)}]
	\item \label{item_bounce0_R} for every diverging sequence $\un$, it holds that:  
	$ \cB{0}{n,\cdot}(X)/\un  \convergence{\Prob_x\text{-}\ucp} 0 $, 
	\item $ \cB{1}{n,\cdot}(X)/\sqn \convergence{\Prob_x\text{-}\ucp}
	2 \sqrt{\frac{2}{\pi}} \loct{X}{0}{} $, \label{item_bounce1_R}
	\item $ \cB{2}{n,\cdot}(X)/n \convergence{\Prob_x\text{-}\ucp} \frac{\rho}{2} \loct{X}{0}{} $, \label{item_bounce2_R}
	\item $ \cD{j}{n,\cdot}(X)/n^{j/2}$ has the same limit as $ \cB{j}{n,\cdot}(X)/n^{j/2}$ for $j=1,2$,
	\label{rmk_main_Rwt}
\end{enumerate}
where as per~\eqref{eq_SDE_X}, $\frac{\rho}{2} \loct{X}{0}{t}= \int_{0}^{t} \indicb{X_s = 0} \vd s $, for all $t\ge 0 $.
\end{theorem}

\begin{proof}[Proof of Theorem~\ref{thm_main_R}]
Let $Y$ be a sticky Brownian motion of stickiness parameter $\rho $ defined on the probability space $\bP'_x = (\Omega',\bF',\process{\bF'_t},\Prob'_x) $ such that 
$\Prob'_x$-\as, $Y_0 = x $ and $Y$ is $\process{\bF'_t}$-adapted. 
We note that $X$ and $|Y|$ are equal in law (see discussion before the statement) and that
\begin{equation}
	\label{eq_proof_refl_BCZ_relations}
	\begin{aligned}
		\cB{0}{n,\cdot}(|Y|) = \cB{0}{n,\cdot}(Y)+\cC{0}{n,\cdot}(Y),\\ 
		\cB{1}{n,\cdot}(|Y|) = \cD{j}{n,\cdot}(Y) + \cB{0}{n,\cdot}(|Y|),\\
		\cB{2}{n,\cdot}(|Y|) = \cD{j}{n,\cdot}(Y) + \cB{0}{n,\cdot}(|Y|),
	\end{aligned}
\end{equation}
From Theorems~\ref{thm_main} and~\ref{thm_main_R0}, for every diverging sequence $(\un)_n $,
\begin{equation}
	\begin{aligned}
		\cB{0}{n,t}(|Y|) &\convergence{\Prob'_x} 0,
		& \frac{1}{\sqn} \cB{1}{n,t}(|Y|) &\convergence{\Prob'_x} 4 \sqrt{\frac{2}{\pi}} \loct{Y}{0}{t},
		& \frac{1}{n} \cB{2}{n,t}(|Y|) &\convergence{\Prob'_x} \rho \loct{Y}{0}{t}.
	\end{aligned}
\end{equation} 
This and the fact that $\loct{Y}{0}{t}= \loct{Y}{0}{t} \indicb{\loct{Y}{0}{t}>0}$ and $\loct{|Y|}{0}{t}=2\loct{Y}{0}{t} $ imply the following convergences 
\begin{align}
	&\cB{0}{n,t}(|Y|) \convergence{\Prob'_x} 0,
	\quad \frac{\indicb{\loct{|Y|}{0}{t}>0}}{\sqn} \frac{\cB{1}{n,t}(|Y|)}{\loct{|Y|}{0}{t}} \convergence{\Prob'_x} 2 \sqrt{\frac{2}{\pi}},
	\\ &\frac{\indicb{\loct{|Y|}{0}{t}>0}}{n} \frac{\cB{2}{n,t}(|Y|)}{\loct{|Y|}{0}{t}} \convergence{\Prob'_x} \frac{\rho}{2} 
\end{align} 
and 
\begin{equation}
	\frac{\indicb{\loct{|Y|}{0}{t} = 0}}{\sqn} \cB{1}{n,t}(|Y|) \convergence{\Prob'_x} 0, \quad
	\frac{\indicb{\loct{|Y|}{0}{t} =0}}{n} \cB{2}{n,t}(|Y|) \convergence{\Prob'_x} 0.
\end{equation}
It is known that, in the case of constant limits, convergence in probability is equivalent to convergence in law (see~\cite[Theorem~25.2]{Billingsley1968_probability} and the discussion thereafter).
And since $|Y|= X $ in law (which entails that $(|Y|,\loct{|Y|}{0}{})=(X,\loct{X}{0}{}) $ in law), we the latter convergences hold with $X$ instead of $|Y|$ and $\Prob_x$ instead of $\Prob'_x$.  
Therefore, using again that ${\loct{X}{0}{t}} = {\loct{X}{0}{t}} \indicb{\loct{X}{0}{t}>0}$, we obtain 
\begin{align}
	&\cB{0}{n,t}(X) \convergence{\Prob_x} 0, \quad
	\frac{\indicb{\loct{X}{0}{t}>0}}{\sqn} \cB{1}{n,t}(X)  \convergence{\Prob_x} 2 \sqrt{\frac{2}{\pi}} {\loct{X}{0}{t}},\\
	& \frac{\indicb{\loct{X}{0}{t}>0}}{n} \cB{2}{n,t}(X) \convergence{\Prob_x} \frac{\rho}{2} {\loct{X}{0}{t}},
\end{align} 
and 
\begin{equation}
	\frac{\indicb{\loct{X}{0}{t}=0}}{\sqn} \cB{1}{n,t}(X)  \convergence{\Prob_x} 0,
	\quad \frac{\indicb{\loct{X}{0}{t}=0}}{n} \cB{2}{n,t}(X) 
	\convergence{\Prob_x} 0.
\end{equation} 
Thus, we have
\begin{equation}
	\label{eq_proof_refl_BCZ_last_eq}
	\begin{aligned}
		\cB{0}{n,t}(X) &\convergence{\Prob_x} 0,
		& \frac{1}{\sqn} \cB{1}{n,t}(X) &\convergence{\Prob_x} 2 \sqrt{\frac{2}{\pi}} \loct{X}{0}{t},
		& \frac{1}{n} \cB{2}{n,t}(X) &\convergence{\Prob_x} \frac{\rho}{2} \loct{X}{0}{t}.
	\end{aligned}
\end{equation} 
From Lemma~\ref{lem_ucp_convergence_condition}, the convergences are uniform in time, in $\Prob_x$-probability ($\Prob_x$-ucp).
This proves~\ref{item_bounce0_R}, \ref{item_bounce1_R}, \ref{item_bounce2_R}.
From~\eqref{eq_proof_refl_BCZ_relations} and~\eqref{eq_proof_refl_BCZ_last_eq}, we  infer~\ref{rmk_main_Rwt}. 
This completes the proof. 
\end{proof}

\begin{corollary}
We consider the setting of Theorem~\ref{thm_main_R} and $\mathcal H_t $ be the event $\{\tau^{X}_0<t\} = \{\exists\; s \in [0,t) \colon  X_s = 0\}$.
The statistic 
\begin{equation}
	\begin{aligned}
		\varrho_n(X) &= \left(4  \sqrt{\frac 2 \pi} \right) \frac{1}{\sqn}\frac{\cD{2}{n,\cdot}(X)}{ \cD{1}{n,\cdot}(X)}
	\end{aligned}
\end{equation}
is a consistent estimators of $\rho $, conditionally on the event $\mathcal H_t $, i.e., $\varrho_n(X) \lra \rho $, in $\Prob^{\mathcal H_t}_x $-probability, as $n\lra \infty $.
\end{corollary}

\begin{proof}
Similar to the proof of Proposition~\ref{prop_estimation}, using
$ \cD{1}{},\cD{2}{} $ instead of $\cC{1}{},\cC{2}{} $
and Theorem~\ref{thm_main_R}\ref{rmk_main_Rwt} instead of Theorem~\ref{thm_main}.
\end{proof}

\subsection{Crossings and bouncings of sticky Itô diffusions}

In this section we generalize Theorem~\ref{thm_main} and Theorem~\ref{thm_main_R0} to sticky Itô diffusions.
The generalization procedure is analogue to the ones in, e.g.,  \cite{Anagnostakis2022,AnagnostakisM2023,Aza89,Jac98,Maz19}.

We consider the system 
\begin{equation}
\label{eq_stickySDE}
\begin{dcases}
	\vd X_t = \indicb{X_t \not = 0}\left(\mu_t \vd t + \sigma(X_t) \vd \W{t}\right),\\
	\indicb{X_t = 0} \vd t = \rho \vd \loct{X}{0}{t},
\end{dcases}
\end{equation}
with $\W{}$ a standard Brownian motion, defined on the probability space $\mathcal P_x = (\Omega, \bF, \process{\bF_t},\Prob_x) $ such that $\Prob_x$-\as,
$X_0=x$. 
%For simplicity, assume that $\process{\bF_t}$ is the natural filtration. 
We suppose $(\mu,\sigma) $ satisfy the following conditions.  
\begin{enumerate}
[label={\upshape (c\arabic*)}] 

\item \label{item_SID_regularity}
The diffusion coefficient $\sigma$ is strictly positive and $C^{1} $.

\item \label{item_SID_weak} Weak existence and uniqueness in law holds for
the system
\begin{equation}
	\label{eq_transform}
	\begin{dcases}
		\vd X'_t = \sigma(X'_t) \indicb{X'_t \ne 0} \vd W_t, \\
		\indicb{X'_t = 0} \vd t = \rho \vd \loct{X'}{0}{t},
	\end{dcases}
\end{equation} 
on $\mathcal P_x$, such that $\Prob_x$-\as, $X'_0=x$. 

\item \label{item_SID_girsanov} The laws of the solutions $X$, $X'$ of~\eqref{eq_stickySDE}, \eqref{eq_transform} with the same initial condition ($X_0 = X'_0 = x $ a.s.)
are locally equivalent in the sense of~\cite[Definition~III.3.2]{jacod2003limit}. 
Basically, this means 
\begin{equation}
	\begin{aligned}
		\forall t &\ge 0,
		& \text{\it Law}(X|_{[0,t]}) & \sim \text{\it Law}(X'|_{[0,t]}).
		%& \text{on}&
		%& &(\Omega,(\bF_s)_{s\le t}).
	\end{aligned}
\end{equation}
Note that the processes are not necessarily defined on the same probability space.

\item \label{item_SID_statespace} The state-space of $X' $ that solves~\eqref{eq_transform} is an open interval $J$ of $\IR $, i.e., $X'$ has only inaccessible boundaries
(see, e.g.,~\cite[Section~33]{Kal}). 

\end{enumerate}

\begin{remark}
Condition~\ref{item_SID_weak} is equivalent to the weak existence and uniqueness for the classical SDE
\begin{equation}
	\label{eq_transform2}
	\vd X''_t = \sigma(X''_t) \vd W_t.
\end{equation}
This is a consequence of, e.g.,~\cite[Theorem~4.1]{Anagnostakis2022} and~\cite[Proposition~4.2]{Anagnostakis2022}.
\end{remark}

Under conditions~\ref{item_SID_regularity}--\ref{item_SID_statespace}, the following results hold. 

\begin{theorem}
\label{thm_main_SID}
Let $\cC{0}{},\cC{1}{},\cC{2}{},\cB{0}{},\cB{1}{},\cB{2}{},
\cD{1}{},\cD{2}{} $ be the quantities defined in 
Section~\ref{sec_intro} for the process $X$.
Then, 
\begin{enumerate}
	[label={ (\roman*)}]
	\item \label{item_strictX} for every diverging sequence $(\un)_n$, 
	\begin{equation}
		\bigbraces{\cB{0}{n,\cdot}(X)/\un}, \bigbraces{\cC{0}{n,\cdot}(X)/\un}  \convergence{\Prob_x\text{-}\ucp} 0,
	\end{equation} 
	\item \label{item_strictX2} $ \bigbraces{\cB{1}{n,\cdot}(X) /\sqn}, \bigbraces{\cC{1}{n,\cdot}(X) /\sqn}, \bigbraces{\cD{1}{n,\cdot}(X) /\sqn} \convergence{\Prob_x\text{-}\ucp}
	\frac{4}{\sigma(0)} \sqrt{\frac 2 \pi} \loct{X}{0}{} $, \label{item_largeX}
	\item \label{item_strictX3} $ \bigbraces{\cB{2}{n,\cdot}(X) /n}, \bigbraces{\cC{2}{n,\cdot}(X) /n},
	\bigbraces{\cD{2}{n,\cdot}(X) /n}
	\convergence{\Prob_x\text{-}\ucp} \rho \loct{X}{0}{} $, \label{item_staysX}
\end{enumerate}
where as per~\eqref{eq_stickySDE}, $\rho \loct{X}{0}{t}= \int_{0}^{t} \indicb{X_s = 0} \vd s $, for all $t\ge 0 $.
\end{theorem}

In the proof we reduce ourselves to the case of sticky Brownian motion.  
By the local equivalence in law~\ref{item_SID_girsanov}, it suffices to prove the result for equation~\eqref{eq_transform}. 
We then further reduce to the case of sticky Brownian motion by Girsanov lemma and Lamperti transform. 

\begin{proof}[Proof of Theorem~\ref{thm_main_SID}]
Let us first assume that the diffusion coefficient and its derivative satisfy the following conditions: there exists a real constant $\delta >0 $ such that 
\begin{equation}
	\label{eq_temporary_condition_sigma}
	\begin{aligned}
		\forall x&\in \IR: 
		&\delta &\le \sigma(x)\le 1/\delta, 
		&\bigabsbraces{\sigma'(x)} &\le 1/\delta .
	\end{aligned}
\end{equation}
We note that under this condition, the state-space $J$ is $\IR $. 
We consider $X' $ to be the process, solution to~\eqref{eq_transform}, defined on $\mathcal P_x$ and $\process{\bF_t}$ to be the natural filtration generated by $X'$. 
By the Definition~\ref{def:ucp} of ucp convergence, we can fix $T>0$ and consider that all processes are restricted to the time interval $[0,T]$. 

Let $\mathcal{E}(\theta) $, $\theta $ be the processes defined by 	\begin{equation}\label{Girsanov_RN_derivative}
	\begin{aligned}
		\mathcal{E}_t(\theta)  &= 	\exp \Bigbraces{\int_{0}^{t} \theta_s \vd W_s - \frac{1}{2} \int_{0}^{t} \theta^2_s \vd s},
		&\theta_t &= \frac12 \sigma'(X_t), 
		& t&\ge 0.
	\end{aligned}
\end{equation}
Since $\sigma' $ is bounded, 
from the Girsanov theorem~\cite[Lemma~4.4]{Anagnostakis2022}, there exists a measure $\Qrob_x$ such that the process $(X'_t, \Wth{t})_{t\in [0,T]} $ jointly solves a drifted version of~\eqref{eq_transform}
\begin{equation}
	\label{eq_proof_SID_X}
	\begin{dcases}
		\vd X'_t	= \frac{1}{2}\sigma(X'_t)\sigma'(X'_t) \indicb{X'_t \ne 0} \vd t + \sigma(X'_t) \indicb{X'_t \ne 0} \vd \Wth{t}, \\
		\indicb{X'_t = 0}	\vd t = \rho \vd \loct{X'}{0}{t},
	\end{dcases}		
\end{equation}		
where $(\Wth{t})_{t\in [0,T]}$ is the $\Qrob_x$-Brownian motion defined by $W^{\theta}=  \W{} - \int_{0}^{\cdot} \theta_s \vd s$.
Moreover the Radon-Nikodym derivative 
$\vd  \Qrob_x/ \vd \Prob_x |_{\bF_T} = \mathcal{E}_T(\theta)$.  

Let $S$ be the function defined by
\begin{equation}
	\label{eq_S_sigma_def}
	\begin{aligned}
		S(y) &=
		\int_{0}^{y} \frac{1 }{\sigma(u)} \vd u,
		& y& \in \mathbb R. 
	\end{aligned}
\end{equation}
The function $S $ is strictly increasing, $S(0)=0 $,
$S \in C^{2}(\IR) $, and 
\begin{equation}
	\begin{aligned}
		S'(y) &= \bigbraces{\sigma(y)}^{-1},
		&
		S''(y) &= - \sigma'(y) \bigbraces{\sigma(y)}^{-2},
		& y&\in \mathbb R. 
	\end{aligned}
\end{equation}
Let $\Xp{} := S(X') $, then from It\^o--Tanaka formula (e.g.,~\cite[Lemma~4.5]{Anagnostakis2022}),~\eqref{eq_proof_SID_X} and the fact that $S(0)=0 $, it holds that $(\Xp{}, \Wth{})$ solves
\begin{equation}\label{eq_proof_sX_solution}
	\begin{dcases}
		\vd \Xp{t} = \frac{1}{\sigma(X'_t)} \vd X'_t - \frac{1}{2}  \frac{\sigma'(X'_t)}{\sigma^2(X'_t)} \vd \qv{X'}_t
		=  \indicb{\Xp{t} \ne 0}  \vd \Wth{t} ,  \\
		\indicb{\Xp{t} = 0}	\vd t = \rho \vd \loct{X'}{0}{t} =  
		\bigbraces{\rho \sigma(0)} \vd \loct{Y'}{0}{t}.
	\end{dcases}		
\end{equation}		
From characterization~\ref{item_sticky_P2_rev}, the process $Y'$ is a sticky Brownian motion under $\Qrob_x $ of stickiness parameter $\rho \sigma(0) $ and 
Theorems~\ref{thm_main}, \ref{thm_main_R0} do apply to $\Xp{} $ for the probability $\Qrob_x$.
In particular, it holds that
\begin{enumerate}
	[label={ (\roman*)}]
	\item $\bigbraces{\cB{0}{n,\cdot}(\Xp{})/\un}, \bigbraces{\cC{0}{n,\cdot}(\Xp{})/\un} \convergence{\Qrob_x \text{-}\ucp} 0 $, 
	\item $\bigbraces{\cB{1}{n,\cdot}(\Xp{}) /\sqn},
	\bigbraces{\cC{1}{n,\cdot}(\Xp{}) /\sqn}, \bigbraces{\cD{1}{n,\cdot}(\Xp{}) /\sqn} \convergence{\Qrob_x \text{-}\ucp} 4 \sqrt{\frac 2 \pi} \loct{\Xp{}}{0}{} $, 
	\item $\bigbraces{\cB{2}{n,\cdot}(\Xp{}) /n},
	\bigbraces{\cC{2}{n,\cdot}(\Xp{}) /n}, 
	\bigbraces{\cD{2}{n,\cdot}(\Xp{}) /n} \convergence{\Qrob_x \text{-}\ucp} \rho \sigma(0) \loct{\Xp{}}{0}{} $. 
\end{enumerate}
Since $S$ is strictly increasing and $S(0)=0 $, then,
\begin{enumerate}
	[label={ (\roman*)}]
	\item for all $n,t $ and $j\in \{0,1,2\} $, $\cC{j}{n,t}(\Xp{}) = \cC{j}{n,t}(X') $, $\cB{j}{n,t}(\Xp{}) = \cB{j}{n,t}(X') $,
	\item for all $n,t $ and $j\in \{1,2\} $, 
	$\cD{j}{n,t}(\Xp{}) = \cD{j}{n,t}(X') $,
	\item for all $i,n $, $U^{n}_{i}(\Xp{})=U^{n}_{i}(X') $, 
\end{enumerate}
and $\Qrob_x $-\as, $\loct{X'}{0}{t} = \sigma(0)\loct{\Xp{}}{0}{t}$. 
From equivalence of measure $\Prob_x \sim \Qrob_x$ when restricted to $\bF_T$, and the fact that
$\mathcal E(\theta)_T $ is an exponential martingale (even square integrable), we deduce that the previous convergences hold also
under $\Prob_x $. The proof is thus completed for~\eqref{eq_transform} under the boundedness assumptions on the diffusion coefficient and its derivative. 

Let us consider a standard localization argument to go from bounded to unbounded  $\sigma,1/\sigma,\sigma' $, 
by considering $X'^{\tau'_m} = \process{X'_{t\wedge \tau'_m}} $.
More precisely, let $\process{\bF_t}$ be the filtration generated by $X'$,
let $(K_m)_m $ be an increasing sequence of compacts of $J$ such that $\bigcup_m K_m = J $, and let $(\tau'_m)_m $ be the sequence of stopping times defined by 
$\tau'_m = \inf\{t\ge 0:\; X'_t \not\in K_m \}$ and $X'^{\tau'_m} $ be the stopped process, defined by $X'^{\tau'_m} = \process{X'_{t\wedge \tau'_m}} $.
Condition~\ref{item_SID_statespace} ensures that 
\begin{equation}
	\label{eq_convergence_ST}
	\lim_{m\ra \infty} \Prob_x(t > \tau'_m)
	=
	0
\end{equation}
for all $t\geq 0$. 
For a %$\process{\bF_t} $-adapted 
process $Z$ and all $t\geq 0$, let $(f_n(Z,t))_n $ be a sequence of statistics such that $f_n(Z,t) $ is $\bF_t $-measurable. 
These quantities represent the difference between the statistics of interest and their limits.
We also observe that
\begin{equation}
	\begin{aligned}
		\Prob_x (|f_n(X',t)|>\varepsilon) 
		&\le \Prob_x(|f_n(X',t)|>\varepsilon;\; t\le \tau'_m) 
		+ \Prob_x(t > \tau'_m)
		\\
		&= \Prob_x(|f_n(X',t \wedge \tau'_m)|>\varepsilon;\; t\le \tau'_m) 
		+ \Prob_x(t > \tau'_m)
	\end{aligned}
\end{equation}
We use the previous part to pass to the limit as $n\to \infty$. 
Next, we let $m \lra \infty $.
From~\eqref{eq_convergence_ST}, this proves the result for $X' $.

From~\ref{item_SID_girsanov}, these convergences also hold in $\Prob_x$-probability for the solution
to~\eqref{eq_stickySDE}.

From Lemma~\ref{lem_ucp_convergence_condition}, the convergence is also uniform in time, in $\Prob_x$-probability ($\Prob_x$-ucp).
This completes the proof. 
\end{proof}

\begin{corollary}
\label{cor_estimation_SID}
Let $\mathcal P_x =(\Omega, \bF, \process{\bF_t}, \Prob_x)$ be a filtered probability space and let $X$ be an $\process{\bF_t}$-adapted process solution to \eqref{eq_stickySDE} on $\mathcal P_x$. 
The statistic defined by % CODE: formule ligne 418
\begin{equation}
	\begin{aligned}
		\varrho_n(X) &= \left(\frac{ 4}{\sigma(0)} \sqrt{\frac 2 \pi} \right) \frac{1}{\sqn}\frac{\cC{2}{n,t}(X)}{ \cC{1}{n,t}(X)},
		& n&\in \IN,
	\end{aligned}
\end{equation}
is a consistent estimators of $\rho $, conditionally on the event $\mathcal H_t =\{\tau^{X}_0 < t\}= \{\exists\; s \in [0,t) \colon  X_s = 0\}$.
This means that (see~Definition~\ref{def_conditional_convergence_in_probability}),  $\varrho_n(X) \lra \rho $, in $\Prob^{\mathcal H_t}_x $-probability, as $n\lra \infty $.
\end{corollary}

\begin{proof}
The proof is the same as the proof of of Proposition~\ref{prop_estimation}, using
Theorem~\ref{thm_main_SID} instead of Theorem~\ref{thm_main}. 
We should also justify that $\Prob_x$-\as, $\mathcal H_t = \{ \loct{X}{0}{t} > 0 \} $. 
For the latter it suffices to remark that, as seen in the proof of Proposition~\ref{prop_estimation}, this is  
the case for the sticky Brownian motion, so using the notations of the proof of Theorem~\ref{thm_main_SID}, for $Y$ under $\Qrob_x $.
From~\cite[Exercise~VI.1.23]{RevYor}, this is the case for $X = S^{-1}(Y)$ under $\Qrob_x $.
From the equivalence of probability measures, this is also the case for $X$ under $\Prob_x $.
This completes the proof. 
\end{proof}

\begin{remark}
Corollary~\ref{cor_estimation_SID} holds by replacing $(\cC{1}{n,t}(X),\cC{2}{n,t}(X)) $ 
%with $(\cB{1}{n,t}(X),\cB{2}{n,t}(X))$ or 
with $(\cD{1}{n,t}(X),\cD{2}{n,t}(X))$.
\end{remark}

\section{Numerical experiments}
\label{sec_numexp}

In this section we aim to numerically establish the properties of the
stickiness parameter estimator $\varrho_n$ devised in Proposition~\ref{prop_estimation}.
We compare it with the stickiness parameter estimator $\varrho^{(0)}_n $ of
\cite{AnagnostakisM2023} (recalled in~\eqref{eq_first_stickiness_estimator}) in term of 
variance and convergence rate.

We simulate sample paths of the sticky Brownian motion of stickiness parameter $\rho=1 $. 
For this, we use the STMCA approximation scheme with the tuned grid defined in~\cite[Section~2.3]{anagnostakis2023general} of size criteria $h=0.005 $ (the nomenclature {\it tuned grid} and {\it size criteria} is the one used in~\cite{anagnostakis2023general}, it is a grid adapted to the diffusion we aim to approximate).  
This ensures good convergence properties of the scheme~\cite[Corollary~2.5]{anagnostakis2023general}. 
We consider initial value $X_0 = 0 $ and time horizon $T=1 $. 

For each generated path $\wt X^{(i)} $ we consider the Monte Carlo (MC) estimators of
$\rho $ 
\begin{equation}
\begin{aligned}
	\mu_{\MC}&:= \frac{1}{N_{\MC}} \sum_{i=1}^{N_{\MC}} \varrho_{n}(\wt X^{(i)}), 
	&
	\mu^{(0)}_{\MC}&:= \frac{1}{N_{\MC}} \sum_{i=1}^{N_{\MC}} \varrho_{n}^{(0)}(\wt X^{(i)}),
\end{aligned}
\end{equation}
where $N_{\MC} $ is the Monte Carlo simulation size and the estimator $\varrho^{(0)}_n$,
defined in~\eqref{eq_first_stickiness_estimator}, is considered for: test function $g: [x \ra \indicb{0<|x|<5}/10] $ and $\un = n^{\alpha} $ for different values of $\alpha$ close to $0.5$. For Brownian motion $\alpha=0.5$ yields the higher convergence speed~\cite{Jac98}.

\begin{table}[h!]
\centering
\begin{tabular}{c c c c} 
	estimator & $n$ & MC estimation & MC std. estimation\\
	\hline
	$\varrho_n(X) $ & 200000 & 1.030 & 0.0872 \\
	$\varrho^{(0)}_n(X) $ with $\un = n^{0.4}$ & 200000 & 1.062 & 0.3710 \\
	$\varrho^{(0)}_n(X) $ with $\un = n^{0.45}$ & 200000 & 1.039 & 0.2975 \\
	$\varrho^{(0)}_n(X) $ with $\un = n^{0.5}$ & 200000 & 1.019 & 0.3710 \\
	$\varrho^{(0)}_n(X) $ with $\un = n^{0.55}$ & 200000 & 1.014 & 0.1793 \\
	\hline
\end{tabular}
\caption{
	Stickiness parameter estimation for the sticky Brownian motion:
	$\varrho$ vs $\varrho^{(0)} $. 
	Simulation size: $N_{\MC}= 2000 $.
	High values of $\alpha,n $ induce some bias to the approximation due to the grid nature of the STMCA approximation scheme (see discussion in~\cite{anagnostakis2023general}).
}
\label{tab:sample1}
\end{table}

\begin{table}[h!]
\centering
\begin{tabular}{c c c c }
	estimator & $n$ & MC estimation & MC std. estimation\\
	\hline
	$\varrho_n(X) $ & 25000  & 1.018  & 0.149 \\
	$\varrho_n(X) $ & 50000  & 1.012  & 0.133 \\
	$\varrho_n(X) $ & 100000 & 1.012  & 0.121 \\
	$\varrho_n(X) $ & 200000 & 1.011 & 0.077 \\
	$\varrho_n(X) $ & 300000 & 1.012 & 0.079 \\
	\hline
\end{tabular}
\caption{
	Stickiness parameter estimation of the sticky Brownian motion using the estimator $\varrho$:
	variance as function of $n$. 
	Simulation size: $N_{\MC}= 5000 $.
}
\label{tab:sample2}
\end{table}

We also consider the associated Monte Carlo standard deviation (MC std.) estimators:
\begin{align}
\sigma_{\MC} &:= \biggbraces{ \frac{1}{N_{\MC}} \sum_{i=1}^{N_{\MC}} \bigbraces{\varrho_{n} (\wt X^{(i)}) - \mu_{\MC}}^{2}}^{\frac12},
\\
\sigma^{(0)}_{\MC} &:= \biggbraces{ \frac{1}{N_{\MC}} \sum_{i=1}^{N_{\MC}} \bigbraces{\varrho^{(0)}_{n} (\wt X^{(i)}) - \mu^{(0)}_{\MC}}^{2}}^{\frac12}.
\end{align}

Numerical simulations indicate that the choice of the normalizing sequence affects the overall
convergence rate of $\varrho^{(0)}_{n}(X) $.
Comparing to $\varrho^{(0)}_{n}(X) $, the estimator $\varrho_{n}(X) $ has the advantage that it does not depend upon a test function and a normalizing sequence.
Moreover, numerical simulations indicate the following:
\begin{enumerate}
\item The estimator $\varrho $ seems superior to $\varrho^{(0)} $ for any normalizing sequence as it has less variance (see Table~\ref{tab:sample1}, also
compare Table~\ref{tab:sample2} with \cite[Table~2]{Anagnostakis2022}).
\item The Monte Carlo standard deviation decreases as a function of
the sampling frequency $n$ (see Table~\ref{tab:sample2}).
\end{enumerate}

A theoretical analysis of the convergence speed is object of further research.

\appendix 

\section{Proof of Lemma~\ref{lem_transition_asymptotics}}
\label{app_asymptotics}

For reader's convenience, we recall the statement of Lemma~\ref{lem_transition_asymptotics}. Let us also remind that we consider notation~\eqref{def_mn}--\eqref{notaz_m} of Proposition~\ref{prop_SMA_1}.

\begin{lemma0}
Let $X$ be the sticky Brownian motion of stickiness parameter $\rho>0 $, defined on the probability space $(\Omega, \bF, \process{\bF_t}, \Prob_x) $ such that $X_0 = x $, $\Prob_x$-\as.
Let $ (f_n,k_n,g_n,h_n;\,n\in\IN) $ be the functions defined for all $n \in \IN$ and $x\in \IR $ by 
$f_n(x) := \Prob_{x} \left(  X_{1/n} = 0 \right) $, 
$k_n(x) := \indicb{x\not = 0} f_n(x/\sqn) $, 
$g_n(x) := \Prob_{x} \left( x X_{1/n} < 0 \right) $, 
$h_n(x) := \sqn g_n(x/\sqn) $.
It holds that 
\begin{enumerate}
	[label={\upshape (\roman*)}]
	\item  $\lim_{n\ra \infty} f_n(0) = 1 $ and 
	$\lim_{n\ra \infty} \sqn (1- f_n(0)) = 
	{2 \sqrt{2}}/{\rho \sqrt \pi} $,
	\item  the sequence $(k_n)_n $ satisfies~\eqref{eq_connd_gn_aggreg} and $\lim_{n\ra \infty} m_{\sqn \rho}( k_n )  =  2 \sqrt{{2}/{\pi}} $,
	\item  the sequence $(h_n)_n $ satisfies~\eqref{eq_connd_gn_aggreg} and $\lim_{n\ra \infty} m_{\sqn \rho}( h_n )  = 1/{\rho} $.
\end{enumerate}
\end{lemma0}

Let us also recall some useful results.

\begin{lemma}[e.g.,~\cite{BorSal}, p.124]
\label{lemma_sticky_kernel}
The probability transition kernel of the sticky Brownian motion of stickiness parameter $\rho $ with respect to its speed measure $m_{\rho}(\rd y) = \vd y + \rho \delta_0(\rd y)$ is the function defined for all $t>0 $, 
$x,y\in \IR $ by
\begin{equation}
	p_{\rho}(t,x,y) = \frac{e^{-(x-y)^2/2t} - e^{-(|x|+|y|)^2/2t}}{\sqrt{2 \pi t}} 
	+ \frac{1}{\rho} e^{2(|x|+|y|)/\rho + 2t/\rho^2} \erfc \Bigbraces{\frac{|x|+|y|}{\sqrt{2t}} + \frac{\sqrt{2t}}{\rho}}.
\end{equation}
\end{lemma}

\begin{lemma}
[e.g.,~\cite{Anagnostakis2022}, Corollary~2.7]
\label{lemma_time_scaling}
The probability transition kernel of the sticky Brownian motion with respect to its speed measure 
$m_{\rho}(\rd x) = \vd x + \rho \delta_0(\rd x)$ satisfies for
all $c,\rho,t>0 $ and $x,y\in \IR $
\begin{equation}
	p_{\rho}(ct,x,y)m_{\rho}(\rd y)
	= p_{\rho/\sqrt{c}}(t,x/\sqrt{c},y)m_{\rho/\sqrt{c}}(\rd y).
\end{equation}
\end{lemma}

We are now ready to provide the proof of Lemma~\ref{lem_transition_asymptotics}. 
%Let $ (f_n,g_n,k_n,h_n;\, n\in \IN) $ be the functions defined above.
From Lemma~\ref{lemma_time_scaling}, for all $n\in \IN $ and $x\in \IR $, we have that: 
\begin{equation}
\label{eq_def_fghk2}
\begin{aligned}
	f_n(x) &= \left(\sqn \rho\right)  p_{\sqn \rho }(1,\sqn x,0),
	& g_n(x) &=  \int_{\IR} \indicb{xy < 0} p_{\sqn \rho }(1,\sqn x,y) \vd y,\\	
	k_n(x) &= \indicb{x\not = 0} \left(\sqn \rho\right)  p_{\sqn \rho }(1, x,0), 
	& 	h_n(x) &= \sqn  \int_{\IR} \indicb{xy<0} p_{\sqn \rho }(1, x,y) \vd y.
\end{aligned}
\end{equation}

\begin{proof}
[Proof of Item~\ref{lem_fn0_convergence}]
Regarding the first relation, from~\eqref{eq_def_fghk2} and
Lemma~\ref{lemma_sticky_kernel},
\begin{equation}
	f_n(0)
	= \left(\sqn \rho\right)  p_{\sqn \rho }(1,0,0)
	= e^{ 2/n \rho^2} \erfc \Bigbraces{\frac{\sqrt{2}}{\sqn \rho}}.
\end{equation}
Taking the limit as $n\lra \infty $ yields the result. Indeed, it holds that
\begin{equation}
	\lim_{n\ra \infty} f_n(0) = \lim_{y \ra 0} e^{y^{2}} \erfc(y) = 1.
\end{equation}
Regarding the second relation,
from the l'Hôspital's rule,
\begin{equation}
	\lim_{y\ra 0} \frac{1}{y}\left( 1 - e^{y^{2}} \erfc(y) \right)
	= - \lim_{y\ra 0} e^{y^{2}} \erfc'(y) = \frac{2}{\sqrt \pi}.
\end{equation} 
Taking the limit as $n\lra \infty $ yields
\begin{equation}
	\lim_{n\ra \infty} \sqn (1 - f_n(0))
	= \lim_{n\ra \infty} \frac{\sqrt{2}}{ \rho} \left( \frac{\sqrt{2}}{\sqrt{n} \rho} \right)^{-1} (1 - f_n(0)) = \frac{\sqrt{2}}{ \rho} \frac{2}{\sqrt{\pi}}
	= \frac{2}{ \rho}
	\sqrt{\frac{2}{\pi}},
\end{equation}
which is the desired result.		
This completes the proof. 
\end{proof}

\begin{proof}
[Proof of Item~\ref{lem_kn_convergence}]
Since $k_n(0)=0 $, from~\eqref{eq_def_fghk2} and
Lemma~\ref{lemma_sticky_kernel}, we have that
\begin{equation}
	\begin{aligned}
		m_{\sqn \rho}( k_n ) 
		=& \int_{\IR} k_n(x) \vd x
		= \left(\sqn \rho\right) \int_{\IR}  p_{\sqn \rho }(1, x,0) \vd x
		\\ =& \int_{\IR} e^{2|x|/\sqn \rho + 2/n \rho^2} \erfc \biggbraces{\frac{ |x|}{\sqrt{2}} + \frac{\sqrt{2}}{\sqn \rho}} \vd x.
	\end{aligned}
\end{equation}
Regarding the function we integrate, from the Mills' ratio on the Gaussian random variable (see~\cite[p.98]{GriSti}), for some constant $K_{\Mills}>0 $,
we have that
\begin{align}
		\lim_{n\ra \infty} e^{2|x|/\sqn \rho + 2/n \rho^2} \erfc \biggbraces{\frac{ |x|}{\sqrt{2}} + \frac{\sqrt{2}}{\sqn \rho}}
		&=  \erfc \biggbraces{\frac{ |x|}{\sqrt{2}}} ,\\
		\indicb{|x|\ge 1} e^{2|x|/\sqn \rho + 2/n \rho^2} \erfc \biggbraces{\frac{ |x|}{\sqrt{2}} + \frac{\sqrt{2}}{\sqn \rho}} 
		& \le \indicb{|x|\ge 1} K_{\Mills} \frac{\sqrt{2n}\rho}{\sqn \rho|x|+ 2}	\erfc \biggbraces{- \frac{x^{2}}{2}}
		\\ &\le \sqrt{2} K_{\Mills} \erfc \biggbraces{- \frac{x^{2}}{2}}
\end{align}
and that
\begin{equation}
	\indicb{|x|< 1} e^{2|x|/\sqn \rho + 2/n \rho^2} \erfc \biggbraces{\frac{ |x|}{\sqrt{2}} + \frac{\sqrt{2}}{\sqn \rho}}  \le \indicb{|x|<1} e^{2/\rho + 2/\rho^{2}} \erfc \biggbraces{\frac{ |x|}{\sqrt{2}}},
\end{equation}
where both upper bounds are $L^{1}(\IR)$.
Thus, from the dominated convergence theorem,
\begin{equation}
	\lim_{n\ra \infty} m_{\sqn \rho}( k_n )
	= \lim_{n\ra \infty} 
	\int_{\IR} \erfc \biggbraces{\frac{ |x|}{\sqrt{2}}} \vd x
	= 2 \sqrt{\frac{2}{\pi}}.
\end{equation}
This proves the second assertion.

Regarding the first assertion, we observe that
\begin{align}
		&|x| e^{2|x|/\sqn \rho + 2/n \rho^2} 
		\erfc \biggbraces{\frac{ |x|}{\sqrt{2}} + \frac{\sqrt{2}}{\sqn \rho}} 
		\\ &\qquad \le \indicb{|x|\ge 1} K_{\Mills} \frac{\sqrt{2n}\rho |x|}{\sqn \rho|x|+ 2}	
		\erfc \biggbraces{- \frac{x^{2}}{2}}
		+ \indicb{|x|<1} e^{2/\rho + 2/\rho^{2}} \exp \biggbraces{\frac{ |x|}{\sqrt{2}}}
		\\ &\qquad \le \indicb{|x|\ge 1} K_{\Mills} \sqrt{2}	\exp \biggbraces{- \frac{x^{2}}{2}}
		+ \indicb{|x|<1} e^{2/\rho + 2/\rho^{2}} \erfc \biggbraces{\frac{ |x|}{\sqrt{2}}},
\end{align} 
where the upper bound defined a function in $L^{1}(\IR) $.
Hence, from the dominated convergence theorem,
\begin{equation}
	\begin{aligned}
		\lim_{n\ra \infty} {\int_{-\infty}^{+\infty} |x| k_n(x) \vd x}
		&= \int_{\IR} \lim_{n\ra \infty} \biggbraces{|x|e^{2|x|/\sqn \rho + 2/n \rho^2} \erfc \biggbraces{\frac{ |x|}{\sqrt{2}} + \frac{\sqrt{2}}{\sqn \rho}} } \vd x
		\\ &= \int_{\IR} |x| \erfc \biggbraces{\frac{ |x|}{\sqrt{2}}} \vd x = 1.
	\end{aligned}
\end{equation}
With similar arguments, we can prove that
$ m(k^{2}_n) $ converges to some constant $K >0 $.
Also,
\begin{equation}
	\lim_{n\ra \infty} k^{2}_n(\sqn x)
	= \lim_{n\ra \infty} \biggbraces{e^{2|x|/ \rho + 2/n \rho^2} \erfc \biggbraces{\frac{\sqn |x|}{\sqrt{2}} + \frac{\sqrt{2}}{\sqn \rho}}}^{2} = 0.
\end{equation}
From all the above, $(k_n)_n $ satisfies~\eqref{eq_connd_gn_aggreg}.
This completes the proof.
\end{proof}

\begin{proof}
[Proof of Item~\ref{lem_hn_convergence}]
The first assertion is proved with similar arguments as in 
the proof of Item~\ref{lem_kn_convergence}.
Regarding the second assertion, from~\eqref{eq_def_fghk2} and
Lemma~\ref{lemma_sticky_kernel}, we have that
\begin{align}
		&\sqn \indicb{xy<0} p_{\sqn \rho}(1,x,y)
		\\ & \quad =
		\frac{1}{\rho}
		\indicb{xy<0} \Bigbraces{\indicb{|x|+|y|\ge 1}+\indicb{|x|+|y|< 1}} e^{2(|x|+|y|)/\sqn \rho + 2/n \rho^2} \erfc \biggbraces{\frac{ |x|+|y|}{\sqrt{2}} + \frac{\sqrt{2}}{\sqn \rho}}
		\\ & \quad \le 
		2 \biggbraces{\indicb{|x|+|y|\ge 1} K_{\Mills} \frac{\sqrt{2n}\rho (|x|+|y|)}{\sqn \rho(|x|+|y|)+ 2}	e^{- \frac{(|x|+|y|)^{2}}{2}}
			\\ & \qquad\; + \indicb{|x|+|y|<1} e^{2/\rho + 2/\rho^{2}} \erfc \biggbraces{\frac{ |x|+|y|}{\sqrt{2}}}}
		\\ & \quad \le 
		2 \biggbraces{ \indicb{|x|+|y|\ge 1} K_{\Mills} \sqrt{2}	e^{- \frac{(|x|+|y|)^{2}}{2}}
			+ \indicb{|x|+|y|<1} e^{2/\rho + 2/\rho^{2}} \erfc \biggbraces{\frac{ |x|+|y|}{\sqrt{2}}}},
\end{align}
where the upper bound defined a function in $L^{1}(\IR) $.
Hence, from the dominated convergence theorem, 
\begin{align}
		m_{\sqn \rho}( h_n ) =& \sqn \int_{\IR} \int_{\IR} \indicb{xy<0} p_{\sqn \rho}(1,x,y) \vd y \vd x
		\\ =& \frac{1}{\rho} \int_{\IR} \int_{\IR} \indicb{xy<0} e^{2(|x|+|y|)/\sqn \rho + 2/n \rho^2} \erfc \biggbraces{\frac{ |x|+|y|}{\sqrt{2}} + \frac{\sqrt{2}}{\sqn \rho}} \vd y \vd x
		\\ =& \frac{2}{\rho}
		\int_{0}^{\infty} \int_{0}^{\infty} e^{2(x+y)/\sqn \rho + 2/n \rho^2} \erfc \biggbraces{\frac{ x+y}{\sqrt{2}} + \frac{\sqrt{2}}{\sqn \rho}} \vd y \vd x
		\\ &\convergence{} \frac{2}{\rho} \int_{0}^{\infty} \int_{0}^{\infty} \erfc \biggbraces{\frac{ x+y}{\sqrt{2}}} \vd y \vd x
		= \frac{1}{\rho}.
\end{align}
This completes the proof.
\end{proof}

%%%%%%%%%%%% Appendix B
\section{Proof of Lemma~\ref{lem_sticky_refl_principle}}
\label{app_reflection_sticky}
%%%%%%%%%%%%%%%%%%%%%%%%

For reader's convenience, we recall the statement of Lemma~\ref{lem_sticky_refl_principle} using notation~\eqref{eq_txt_hitting_times} for the first hitting time.

\begin{lemma0}[Reflection principle for sticky Brownian motion]
Let $X$ be the sticky Brownian motion of stickiness parameter $\rho>0$ defined on the filtered probability space $\mathcal P_x = (\Omega, \bF, \process{\bF_t},\Prob_x) $ such that $\Prob_x$-\as, $X_0 = x$ (in particular, $X$ is $\process{\mathcal F_t}$-adapted). 
We consider the process $X' $, defined for all $t\ge 0 $ by
\begin{equation}
	\label{eq_lem_sticky_refl_at_0}
	X'_t = 
	\begin{dcases}
		X_t & t< \tau^{X}_{0},\\
		-X_t & t\ge \tau^{X}_{0}.\\
	\end{dcases}
\end{equation}
It holds that $X' $ is also a sticky Brownian motion of stickiness parameter $\rho$.
\end{lemma0}

\begin{proof}
Let $Z$ be the Brownian motion defined in characterization~\ref{item_sticky_P1}, on an extension of the probability space, so that 
\begin{equation}
	\begin{aligned}
		\forall t &\ge 0: 
		&X_t &= Z_{\gamma(t)}		
	\end{aligned}
\end{equation}
with $\gamma$ the time--change defined in~\ref{item_sticky_P1}.
We observe that $\tau^{X}_{0}=\gamma(\tau^{Z}_{0}) $.
Let $Z' $ be the process defined for all $t\ge 0 $ by
\begin{equation}
	Z'_t = 
	\begin{dcases}
		Z_t & t< \tau^{Z}_{0},\\
		-Z_t & t\ge \tau^{Z}_{0}.\\
	\end{dcases}
\end{equation}
From the reflection principle of the Brownian motion $Z$ (see~\cite[Exercise~III.3.14]{RevYor}), we have that $Z'$ is also a Brownian motion. Moreover, $\loct{Z}{0}{}=\loct{Z'}{0}{}$, $\tau^{Z}_0=\tau^{Z'}_0$. Therefore, $A(t)=t+\rho \loct{Z'}{0}{t}$ and $\tau^{X}_0=\gamma(\tau^{Z}_0)=\gamma(\tau^{Z'}_0)$ which, by the definition of $X'$, is equal also to $\tau^{X'}_0$.
From all the above, we show that $X' = \process{Z'_{\gamma(t)}} $. Indeed, on the event $\{t\geq \tau^{X}_0\}$ it holds that 
\begin{equation}
	X_t'= -X_t = - Z_{\gamma(t)} = (-Z)_{\gamma(t)} = Z'_{\gamma(t)},
\end{equation}
and, since $X_t=Z_{\gamma(t)}$, 
\begin{equation}
	X'_t= Z_{\gamma(t)} \indicb{t < \tau^{X}_0} + Z'_{\gamma(t)} \indicb{t\geq \tau^{X}_0} = Z'_{\gamma(t)}.
\end{equation}
To finish the proof, it suffices to observe that \ref{item_sticky_P1_rev} ensures that $X'$ is a sticky Brownian motion of stickiness parameter $\rho$.
\end{proof}

\end{document}